\documentclass[10pt,a4paper]{article}
\usepackage[utf8]{inputenc}
\usepackage[english]{babel}
\usepackage{amsmath,amssymb,amsfonts,amsthm,mathrsfs}
\usepackage{yfonts}
\usepackage{graphicx,color}
\usepackage{epsfig}
\usepackage{bigints}

\usepackage{bbm}
\usepackage{indentfirst}
\usepackage{mathtools}
\usepackage{hyperref}

\usepackage{enumitem}
\usepackage{caption}
\usepackage{subcaption}
\usepackage{pb-diagram}
\usepackage[all]{xy}

\newtheorem{theorem}{Theorem}[section]

\newtheorem{lemma}{Lemma}[section]
\newtheorem{proposition}{Proposition}[section]

\def\mc{\mathcal}

\def\a{\alpha}
\def\b{\beta}
\def\t{\theta}
\def\d{\delta}

\def\g{\gamma}
\def\s{\sigma}
\def\l{\lambda}
\def\p{\partial}

\def\e{\varepsilon}
\def\v{\varphi}
\def\G{\Gamma}
\def\k{\kappa}
\def\o{\omega}
\def\R{\mathbb{R}}



\title{\vskip-2.5cm
{Global multiplicity results in a Moore--Nehari type problem with a spectral parameter}
\thanks{
This paper has been written under the auspices of the Ministry of Science and Innnovation of Spain under Reserach Grant PID2021-123343NB-I00, and  the Institute of Interdisciplinary Mathematics of Complutense University of Madrid. } }

\author{
	\sc Juli\'an L\'opez-G\' omez and Eduardo Mu\~{n}oz-Hern\'andez
	\\
	\small Universidad Complutense de Madrid
	\\
	\small Instituto de Matem\'{a}tica Interdisciplinar (IMI)
	\\
	\small Departamento de An\'alisis Matem\'atico y Matem\'atica
	Aplicada
	\\
	\small  Plaza de las Ciencias 3, 28040   Madrid, Spain
	\\
	\small E-mails: {\tt   julian@mat.ucm.es} and {\tt eduardmu@ucm.es}
	\medskip
	\\
	\sc Fabio Zanolin
	\\
	\small Università degli Studi di Udine
	\\
	\small Dipartimento di Matematica e Informatica
	\\
	\small  Via delle Scienze 206, 33100, Udine, Italy
	\\
	\small E-mail: {\tt fabio.zanolin@uniud.it }
	\bigskip
}
\begin{document}
\maketitle	
\begin{abstract}
This paper analyzes the structure of the set of positive solutions of \eqref{1.1}, where $a\equiv a_h$
is the piece-wise constant function defined in \eqref{1.3} for some $h\in (0,1)$. In our analysis, $\l$ is regarded as a bifurcation parameter, whereas $h$ is viewed as a deformation parameter between the autonomous case when $a=1$ and the linear case when $a=0$. In this paper, besides establishing some of the multiplicity results suggested by the numerical experiments of \cite{CLGT-2024}, we have analyzed the asymptotic
behavior of the positive solutions of \eqref{1.1} as $h\uparrow 1$, when the shadow system of \eqref{1.1} is the linear equation $-u''=\pi^2 u$. This is the first paper where such a problem has been addressed. Numerics is of no help in analyzing this singular perturbation problem because the positive solutions blow-up point-wise in $(0,1)$ as $h\uparrow 1$ if $\l<\pi^2$.

\vspace{0.2cm}

\noindent {\bf Keywords}: Moore--Nehari equation. Multiplicity of positive solutions. Point-wise blow-up to a metasolution.  Spectral parameter.\\
\noindent {\bf MSC 2020}: 34B08, 34B16, 34B18.\\
\end{abstract}

\section{Introduction}\label{sec1}

\noindent This paper studies the one-dimensional second-order superlinear problem
\begin{equation}
\label{1.1}
\left\{
\begin{array}{ll}
-u'' = \lambda u + a(x)u^p,\quad x\in(0,1),\\[1pt]
u(0)=u(1)=0,
\end{array}
\right.
\end{equation}
where $\l\in\mathbb{R}$ is regarded as a bifurcation parameter, $p>1$ is fixed, and $a:[0,1]\to [0,\infty)$ is a non-zero piece-wise constant function. Thus, $a\gneq 0$ in the sense that $a\geq 0$, but $a\neq 0$. The main goal of this paper is to study the global structure of the set of positive solutions of \eqref{1.1}
in terms of the parameter $\l$ and of the nature of the weight function $a(x)$. As $p>1$, by the uniqueness of solution for the Cauchy problem associated to the differential equation of \eqref{1.1}, any positive solution, $u$, of \eqref{1.1} must be strongly positive, in the sense that
\begin{equation}
\label{1.2}
	u(x)>0\;\;\hbox{for all}\; x\in (0,1),\;\;  u'(0)>0\;\; \hbox{and}\;\; u'(1)<0.
\end{equation}
In the special case when $\l=0$ and $p$ is an odd integer, the problem \eqref{1.1} has been already analyzed by Moore and Nehari  \cite{MoNe-1959} for a symmetric piecewise weight function $a(x)$ defined, after a spatial scaling, as
\begin{equation}
\label{1.3}
a(x)= \;
\begin{cases}
1\qquad\hbox{if}\;x\in [0,\tfrac{1-h}{2}]
\cup [\tfrac{1+h}{2},1],\\[1ex]
0\qquad\hbox{if}\;x\in (\tfrac{1-h}{2},\tfrac{1+h}{2}),
\end{cases}
\end{equation}
where $h\in (0,1)$ is fixed. The main result of \cite{MoNe-1959} shows that, for some values of $h\in (0,1)$, the problem \eqref{1.1} possesses, at least, three positive solutions: a first solution which is symmetric (even) with respect to $x=0.5$, and two additional asymmetric solutions, reflected from each other
about $x=0.5$. In recent years, also in the special case when $\l=0$, some additional work has been done by 
Kajikiya et al. \cite{Ka-2018} by considering $h$ as a bifurcation parameter and establishing the existence of a symmetry-breaking from the symmetric solution to a continuum of asymmetric solutions at some critical value $h_*\in(0,1)$ (see also \cite{Ka-2022} for nodal solutions if $p\in(0,1)$, and \cite{Ka-2023} for nodal solutions if $p>1$). Some related problems where a
symmetry-breaking occurs leading to multiplicity of positive solutions can be found in \cite{Ka-2012, Ka-2014,Ta-2009,Ta-2013}.
\par
The main goal of this paper is to ascertain the global topological structure of the set of positive solutions of \eqref{1.1}, $(\l,u)$, regarding $\l\in\R$ as the primary bifurcation parameter. Then, we will ascertain how
changes the underlying global bifurcation diagram in $\l$ as $h$ varies in the interval $(0,1)$. Thus, $h$ is regarded as the  secondary parameter of \eqref{1.1} for the special choice \eqref{1.3}. Throughout this paper, to emphasize the dependence of $a(x)$ on $h\in (0,1)$, we will eventually set
$$
   a\equiv a_h.
$$
According to, e.g.,  Amann and L\'opez-G\'omez \cite{ALG}, under these assumptions, for every $h\in (0,1)$, the set of non-trivial non-negative solutions, $(\l,u)$, of problem \eqref{1.1} has a component, $\mathscr{C}_h^+$, such that $(\pi^2,0)\in \mathscr{C}_h^+$ and
$$
  \mathcal{P}_\l (\mathscr{C}_h^+)= (-\infty,\pi^2],
$$
where $\mathcal{P}_\l$ stands for the $\l$-projection operator,
$$
  \mathcal{P}_\l (\l,u)=\l \;\; \hbox{for all}\;\; (\l,u)\in \R \times \mathcal{C}([0,1];\R).
$$
In particular, \eqref{1.1} admits a positive solution if, and only if,
$\l<\pi^2$. Therefore, the condition $\l <\pi^2$ will be imposed throughout this paper.
Moreover, according to the numerical experiments of Cubillos et al. \cite{CLGT-2024}, for every $h\in (0,1)$, there is a value of the parameter $\l$, denoted by $\l_{h,b}$, which is a  subcritical pitchfork bifurcation value for $\mathscr{C}_h^+$. These numerical experiments also suggest that \eqref{1.1} has three positive solutions for every $\l<\l_{h,b}$, one symmetric and two asymmetric, and that
\begin{equation}
\label{1.4}
\lim_{h\downarrow  0}\l_{h,b}=-\infty\quad \hbox{and}\quad \lim_{h\uparrow 1}\l_{h,b}=\pi^2.
\end{equation}
Although the first limit of \eqref{1.4} is rather natural, because the limiting shadow
problem as $h\downarrow 0$,
\begin{equation}
\label{1.5}
\left\{
\begin{array}{ll}
-u'' = \lambda u+u^p\quad \hbox{in}\;\;(0,1),\\[1pt]
u(0)=u(1)=0,
\end{array}
\right.
\end{equation}
has a unique positive solution for each $\l \in (-\infty,\pi^2)$, which follows by a direct analysis of the phase portrait of the autonomous differential equation
$$
  -u''=\l u+ u^p,
$$
the second limit of \eqref{1.4} is far from expected, because it  entails that, for every $\e>0$ and sufficiently small $1-h$, \eqref{1.1} has three positive solutions for every $\l \in (-\infty,\pi^2-\e]$, while the limiting shadow system of \eqref{1.1} as $h\uparrow 1$ is the linear problem 
\begin{equation}
\label{1.6}
\left\{
\begin{array}{ll}
-u'' = \lambda u\quad \hbox{in}\;\; (0,1),\\[1pt]
u(0)=u(1)=0,
\end{array}
\right.
\end{equation}
which has a positive solution if, and only if, $\l=\pi^2$. Thus, the global multiplicity result as $h\uparrow 1$, for every $\l\in (-\infty,\pi^2)$,  is attributable to  the nonlinear character of $a_hu^p=u^p$ in
$[0,\frac{1-h}{2})\cup (\frac{1+h}{2},1]$, which is a very tiny set as $h\uparrow 1$. Another multiplicity result in a similar vein, caused by spatial heterogeneities, has been documented by L\'opez-G\'omez and Rabinowitz in \cite{LGRab-2016}.
\par
The main goal of this paper is to deliver some analytical proofs of some of these
features and analyze the point-wise behavior of the three positive solutions of \eqref{1.1}
as $h\uparrow 1$. Precisely, our main findings are the following ones:
\begin{enumerate}
\item[{\rm (F1)}] For every $\l<\pi^2$, \eqref{1.1} has a unique positive symmetric solution, $u_\l$, and
$$
		\mathscr{C}_{h,s}^+=\{(\l,u_\l)\;:\; \l \in (-\infty,\pi^2]\}
$$
is a continuous curve such that $\mathscr{C}_{h,s}^+\subset \mathscr{C}_{h}^+$.

\item[{\rm (F2)}] For every $h\in (0,1)$, there exists $\tilde \l=\tilde \l(h)<\pi^2$ such that, for every $\l\leq \tilde \l(h)$, \eqref{1.1} has, at least, three positive solutions.

\item[{\rm (F3)}] For any given $\lambda\in \left(\pi^2/4,\pi^2\right)$, there exists $R_{\lambda}>0$ such that, for every $R>R_{\lambda}$, the problem \eqref{1.1} has, at least, two asymmetric positive solutions for a suitable $h=h(R)\in (0,1)$. Moreover,
$$
    \lim_{R\uparrow +\infty}h(R)=1.
$$

\item[{\rm (F4)}] For every $\l<\pi^2$ and $h\in (0,1)$, let $u_h$ be any positive solution of
\eqref{1.1}. Then,
\begin{equation}
\label{1.7}
	\lim_{h\uparrow 1}u_h(x)=+\infty \;\; \hbox{for all}\;\; x\in (0,1).
\end{equation}
In particular, the components $\mathscr{C}_{h,s}^+$ and $\mathscr{C}_h^+$ blow-up, point-wise in $(0,1)$, as $h\uparrow 1$.
\item[{\rm (F5)}] For every $\a>0$, there are two increasing sequences, $\{h_n\}_{n\geq1}$ in $(0,1)$ and $\{\l_n\}_{n\geq1}$ in $(0,\pi^2)$, such that
$$
\lim_{n\to+\infty}h_n=1, \quad \lim_{n\to+\infty}\l_n=\pi^2,\quad \hbox{and}\;\; \|u_{\l_n}\|_\infty =\a \;\;
\hbox{for all}\;\; n\geq 1,
$$
where $u_{\l_n}$ is the unique symmetric positive solution of \eqref{1.1} for
$(\l,h)=(\l_n,h_n)$, $n\geq 1$. Moreover,
\begin{equation*}
\lim_{n\to+\infty} u_{\l_n}=\alpha\sin(\pi \cdot)\quad  \hbox{in}\;\; \mc{C}^1([0,1];\R).
\end{equation*}
In other words,  $(\pi^2,\a \sin(\pi \cdot))\in \mathscr{C}_1^+$ perturbs into the sequence
$$
  (\l_n,u_{\l_n})\in \mathscr{C}_{h_n,s}^+,\quad n\geq 1.
$$
\end{enumerate}

\noindent The existence and the uniqueness of the symmetric solution established in (F1) will be proven in Section 2. It combines a phase plane analysis with some sharp comparison results. The multiplicity result
established in (F2) will be proven in Section 4 through some phase plane techniques based on a technical device of \cite{LGTZ-2014}. The multiplicity result established in (F3) will be proven in Section 5 combining some sophisticated phase plane techniques with the Mountain Climbing Theorem of \cite{Ho-1952,Wh-1966, Ke-1993}. Finally, the results stated in (F4) and (F5) will be proven in Section 3.
\par
Although, according to \cite{Ka-2018}, it is already known that
$$
   \lim_{h\uparrow 1} \|u_h\|_{\infty}= +\infty \quad \hbox{if}\;\; \l=0,
$$
the main novelty of our Theorem \eqref{3.1} establishes that
$$
   \lim_{h\uparrow 1} u_h(x)= +\infty \quad \hbox{for all}\;\; \l\in (-\infty,\pi^2) \;\;
   \hbox{and}\;\; x\in (0,1),
$$
which is an extremely delicate issue for $\l<0$.  According to (F4) and (F5), the positive solutions of \eqref{1.1} blow-up as $h\uparrow 1$, point-wise in $(0,1)$, if $\l<\pi^2$, while, simultaneously,  any solution $(\l,u)=(\pi^2,\a \sin (\pi x))$ of \eqref{1.6} perturbs
into solutions of \eqref{1.1} with $\l\sim \pi^2$ and $h \sim 1$. Thus, adopting the terminology of \cite{LG-2015}, all the positive solutions of \eqref{1.1} approximate a metasolution of $-u''=\pi^2 u$ as $h\uparrow 1$.  This is the first occasion that such  a phenomenology has been described in the literature.
As a byproduct, it becomes apparent how the pioneering result of Moore and Nehari \cite{MoNe-1959} is actually global in the parameter $\l$.
\par

\setcounter{equation}{0}

\section{Symmetric positive solutions of \eqref{1.1}}
\label{sec2}
\noindent In this section, we will construct the symmetric positive solutions of \eqref{1.1} for the choice \eqref{1.3} for all $h\in (0,1)$. To this purpose, we consider the associated planar system
\begin{equation}
	\label{2.1}
	\left\{
	\begin{array}{ll}
		u'=v,\\[3pt]
		v'=-\lambda u-a(x)u^p,
	\end{array}
	\right.
\end{equation}
in the closed half-plane
\[
\mathscr{H}^+:=\{(u,v)\in {\mathbb R}^2:\; u\geq 0\}.
\]
So, depending on whether $a=1$ or $a=0$, we are lead to analyze one of the following autonomous planar systems
\begin{equation}
	\label{2.2}
	(\mathscr{N})\;\;\left\{
	\begin{array}{ll}
		u'=v,\\[3pt]
		v'=-\lambda u-u^p,
	\end{array}
	\right.\qquad\hbox{or} \qquad
	(\mathscr{L})\;\;\left\{
	\begin{array}{ll}
		u'=v,\\[3pt]
		v'=-\lambda u,
	\end{array}
	\right.
\end{equation}
respectively. Note that we are throughout assuming that $\l <\pi^2$, and that the dynamics of both systems, $(\mathscr{N})$ and $(\mathscr{L})$, change when considering $\l<0$ or $\l\in[0,\pi^2)$.
\par
\subsection{Symmetric solutions for $\l<0$}
\label{sec2.1}
\noindent Suppose $\lambda<0$. Then, the set of equilibria of $(\mathscr{N})$ consists of the origin, $(0,0)$,  plus $(\o,0)$, where $\o :=(-\l)^{\frac{1}{p-1}}$, and the total energy associated to $(\mathscr{N})$ is
\[
\mathscr{E}_{\mathscr{N}}(u,v):=\frac{1}{2}v^2+\frac{\lambda}{2}u^2+\frac{1}{p+1}u^{p+1}.
\]
Thus, $(0,0)$ is a half-right saddle point whose stable and unstable manifolds provide us with
an homoclinic connection of $(0,0)$ corresponding to the level set $\mathscr{E}_{\mathscr{N}}(u,v)=0$, surrounding the equilibrium $(\o,0)$, which is a local center. Note that the level set $\mathscr{E}_{\mathscr{N}}(u,v)=0$ crosses the $u$-axis at the point $(u_{\rm ho},0)$, where
\[
u_{\rm ho}:=\left(\frac{-\lambda(p+1)}{2}\right)^{\frac{1}{p-1}}.
\]
Thus, $(u_{\rm ho},0)$ is the crossing point of the homoclinic connection of the origin
with the positive $u$-axis. The phase portrait of
($\mathscr{N}$) has been sketched in the left figure of Figure \ref{Fig1}.

\begin{figure}[h!]
	\centering
	\includegraphics[scale=0.2]{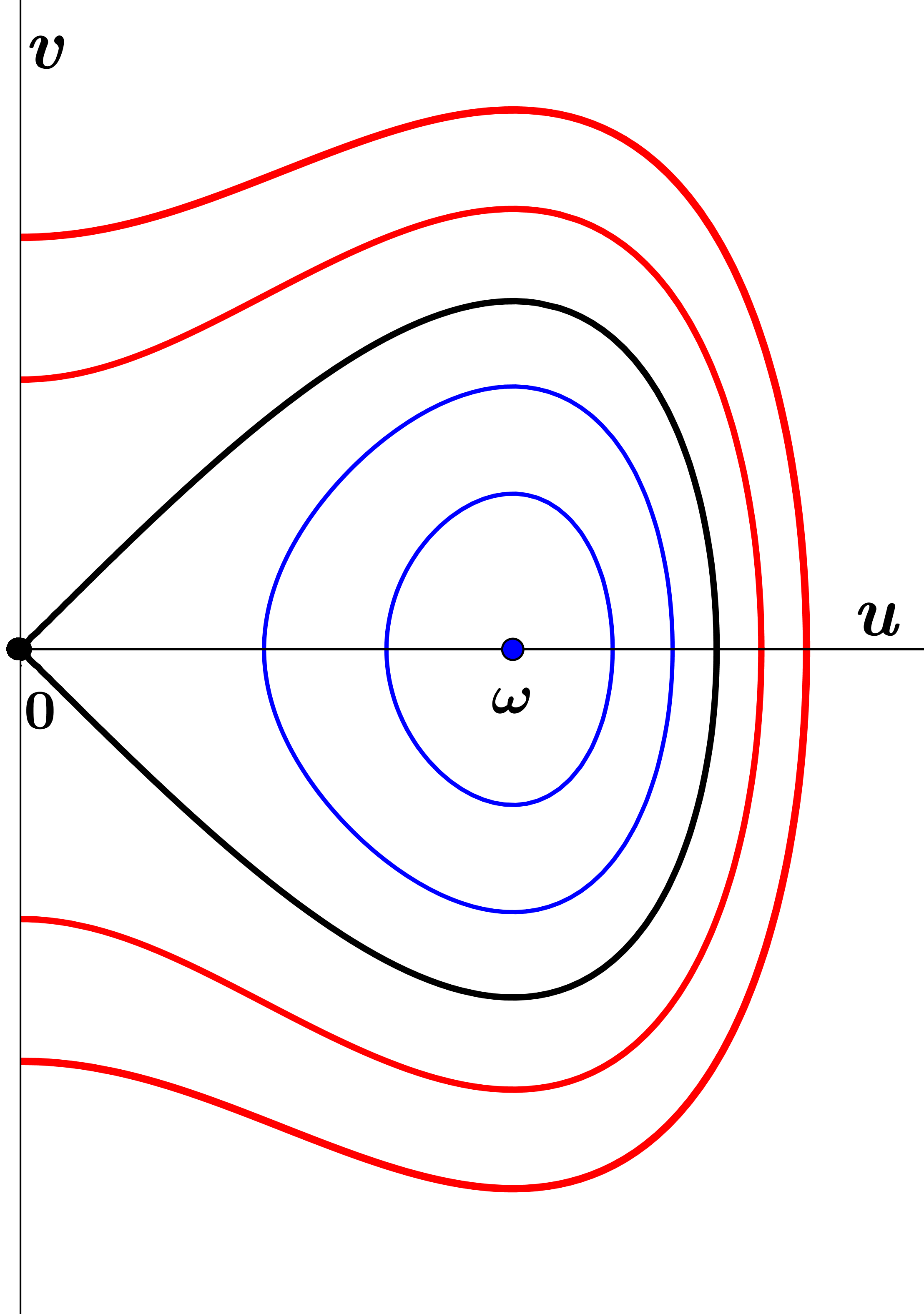}\hspace{2cm}
	\includegraphics[scale=0.2]{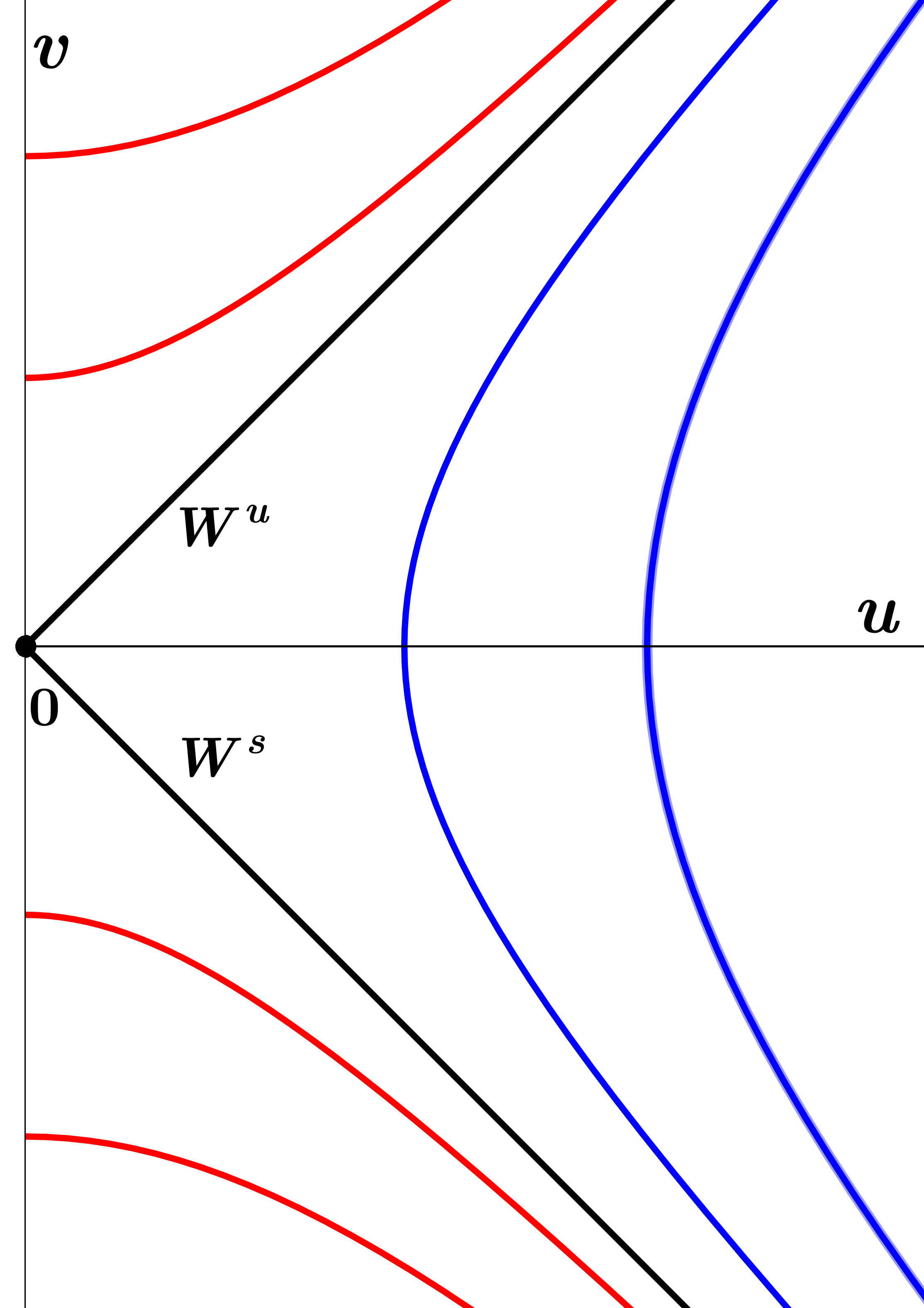}
	\caption{Phase plane diagrams of ($\mathscr{N}$) (left) and ($\mathscr{L}$) (right) when $\l<0$. Integral curves with positive energy are plotted in red, while those with zero energy are plotted
		in black and those with negative energy are colored in blue.}
	\label{Fig1}
\end{figure}
\par
By simply looking at the left picture of Figure \ref{Fig1}, it becomes apparent that, for every $u_0>u_{\rm ho}$, there exists
$$
T:=T_{\mathscr{N}}(\l,u_0)>0
$$
such that the unique solution of ($\mathscr{N}$), $(u(t),v(t))$, satisfying
$$
(u(0),v(0))=(u_0,0),
$$
is defined in $[-T,T]$ and it satisfies
\begin{equation}
	\label{2.3}
	\mathscr{E}_{\mathscr{N}}(u(t),v(t))=\frac{1}{2}v^2(t)
	+\frac{\lambda}{2}u^2(t)+\frac{1}{p+1}u^{p+1}(t)=\frac{\lambda}{2}u_0^2+\frac{1}{p+1}u_0^{p+1}
\end{equation}
for all $t\in[-T,T]$. Moreover,
$$
(u(\pm T),v(\pm T))=(0,\mp v_0),\quad   v_0:= \sqrt{\l u_0^2 +\frac{2}{p+1}u_0^{p+1}},
$$
and a standard  calculation from \eqref{2.3} yields
\begin{equation}
	\label{2.4} T_{\mathscr{N}}(\l,u_0):=\int_0^1\frac{ds}{\sqrt{\l(1-s^2)+\frac{2}{p+1}u_0^{p-1}(1-s^{p+1})}}.
\end{equation}
Thanks to \eqref{2.4}, it is apparent that, for every $\l<0$,  $T_{\mathscr{N}}(\l,u_0)$ is decreasing with respect to $u_0$. Moreover,
\begin{equation}
	\label{2.5}
	\lim_{u_0\downarrow u_{\rm ho}}{T}_{\mathscr{N}}(\l,u_0)=+\infty, \qquad
	\lim_{u_0\uparrow +\infty}{T}_{\mathscr{N}}(\l,u_0)=0.
\end{equation}
\par
As far as ($\mathscr{L}$) is concerned, it is a linear system whose unique equilibrium is the origin, $(0,0)$,  which is a saddle point if $\l<0$. The unstable and stable manifolds, $W^u$ and $W^s$,  are the straight lines
\[
W^u:\;v=\sqrt{-\l}\,u,\qquad W^s:\;v=-\sqrt{-\l}\,u,
\]
respectively, as sketched in the right picture of Figure \ref{Fig1}. The total energy associated to the linear system ($\mathscr{L}$) is
\begin{equation*}
	\mathscr{E}_{\mathscr{L}}(u,v):=\frac{1}{2}v^2+\frac{\lambda}{2}u^2.
\end{equation*}
Thus,  the level set $\mathscr{E}_{\mathscr{L}}(u,v)=0$ consists of $W^u\cup W^s$. Moreover, for every $u_+>0$, the necessary time to link $(u_+,0)$ with an arbitrary point
$$
(u_\ell,v_\ell)\in\left\{(u,v)\,:\,\mathscr{E}_{\mathscr{L}}(u,v)=\tfrac{\l}{2}u_+^2\right\}
$$
is given by
\begin{equation}
\label{2.6}
\begin{split}
{T}_{\mathscr{L}}(\l,u_+,u_\ell) & :=\int_{u_+}^{u_{\ell}}\frac{du}{\sqrt{\l(u_+^2-u^2)}} =  \frac{1}{\sqrt{-\l}}\int_{1}^{\frac{u_{\ell}}{u_+}} \frac{ds}{\sqrt{-1+s^2}}\\ &
=\frac{1}{\sqrt{-\l}}\ln\Big(\tfrac{u_\ell}{u_+}+\sqrt{\Big(\tfrac{u_\ell}{u_+}\Big)^2-1}\,\Big),
\end{split}
\end{equation}
because the function
$$
g(s):=\ln(s+\sqrt{s^2-1}), \qquad s>1,
$$
satisfies $g'(s)=\frac{1}{\sqrt{s^2-1}}$ for all $s>1$.
\par
Consequently, to construct a symmetric positive solution, say $u$, of  \eqref{1.1},  we should have
$$
u'(0.5)=0,\quad u(0.5)=u_+,
$$
for some $u_+>0$, and should proceed as sketched next:
\par
\vspace{0.2cm}
\noindent \emph{Step 1:} Since $g'(s)>0$ for all $s>1$, the function $g$ is invertible. Thus, we should consider the set of points $(u_\ell,v_\ell)$ in the level energy $\mathscr{E}_{\mathscr{L}}(u,v)=\frac{\l}{2}u_+^2$ such that
\[
\frac{u_\ell}{u_+}=g^{-1}\left(\frac{h}{2}\sqrt{-\l}\right)=:C>1,
\]
because, for this choice, \eqref{2.6} guarantees that
$$
{T}_{\mathscr{L}}(\l,u_+,u_\ell)=\frac{h}{2}.
$$
Note that
$$
\frac{u_\ell}{C}=u_+=\sqrt{\frac{v_\ell^2}{\l}+u_\ell^2}
$$
implies
$$
v_\ell =\pm \sqrt{-\l(1-C^{-2}})\,u_\ell.
$$
Therefore, the points $(u_\ell,v_\ell)$ lying in the straight line
$$
v=-\sqrt{- \l (1-C^{-2})}\, u,
$$
need a time $\frac{h}{2}$ to reach $(u_+,0)$, and, by symmetry, they as well need another $\frac{h}{2}$ unities of time to link $(u_+,0)$ with
$$
(u_\ell,\sqrt{-\l(1-C^{-2}})\,u_\ell)
$$
along the trajectory $\mathscr{E}_{\mathscr{L}}(u,v)=\frac{\l}{2}u_+^2$.
\par
\vspace{0.2cm}
\noindent \emph{Step 2:} According to Step 1, any positive solution of \eqref{1.1} starting at some point $(0,v_0)$, traveling along the trajectory  $\mathscr{E}_{\mathscr{N}}=\frac{v_0^2}{2}$ of the nonlinear system ($\mathscr{N}$) until reaching some point
$$
(u_\frac{1-h}{2},v_\frac{1-h}{2})\equiv (u_\ell,-\sqrt{-\l(1-C^{-2})}\,u_\ell)
$$
at $x=\frac{1-h}{2}$ provides us with a symmetric solution of \eqref{1.1}, as sketched in Figure \ref{Fig2}.

\begin{figure}[h!]
	\centering
	\includegraphics[scale=0.2]{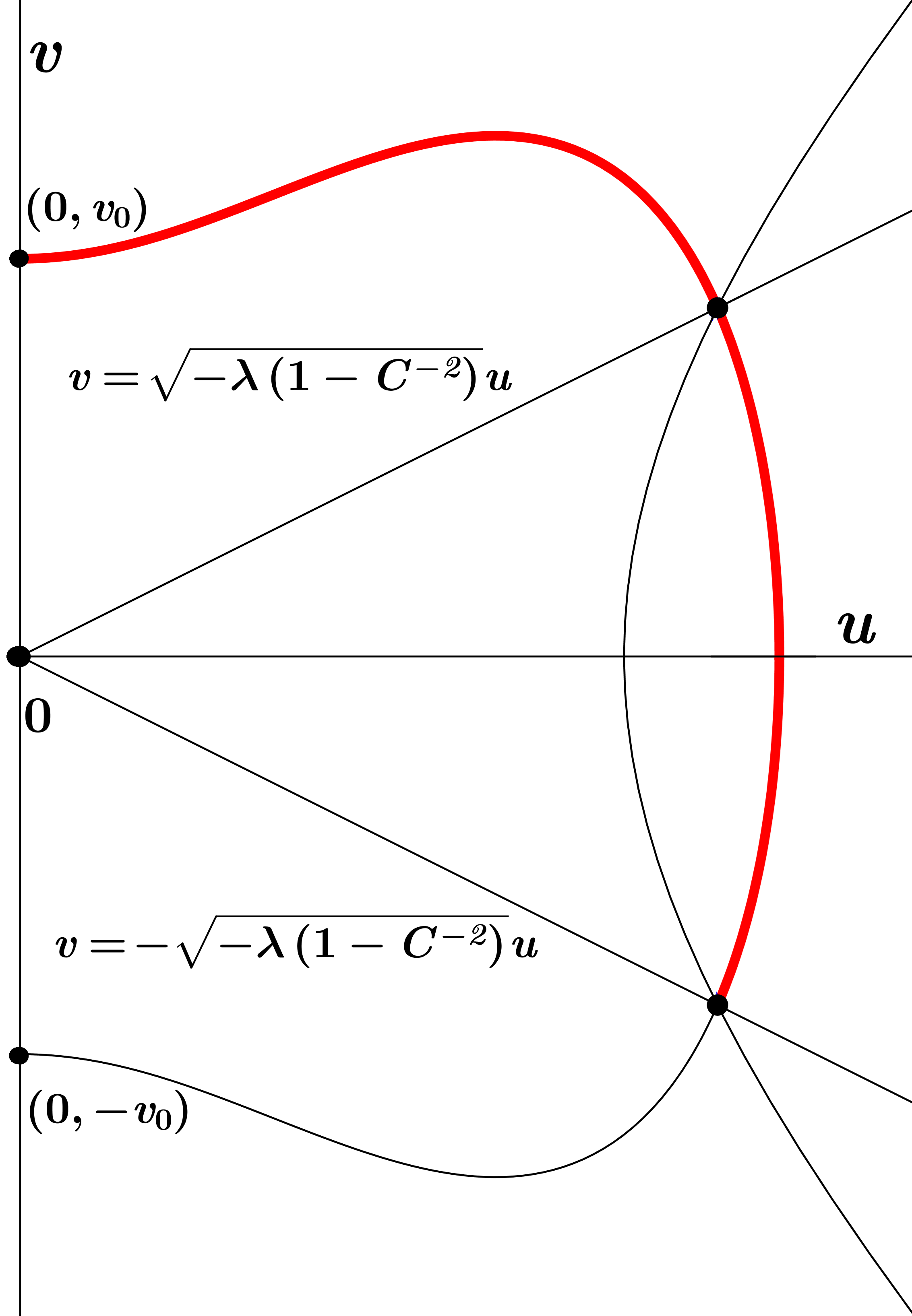}\;\;
	\includegraphics[scale=0.2]{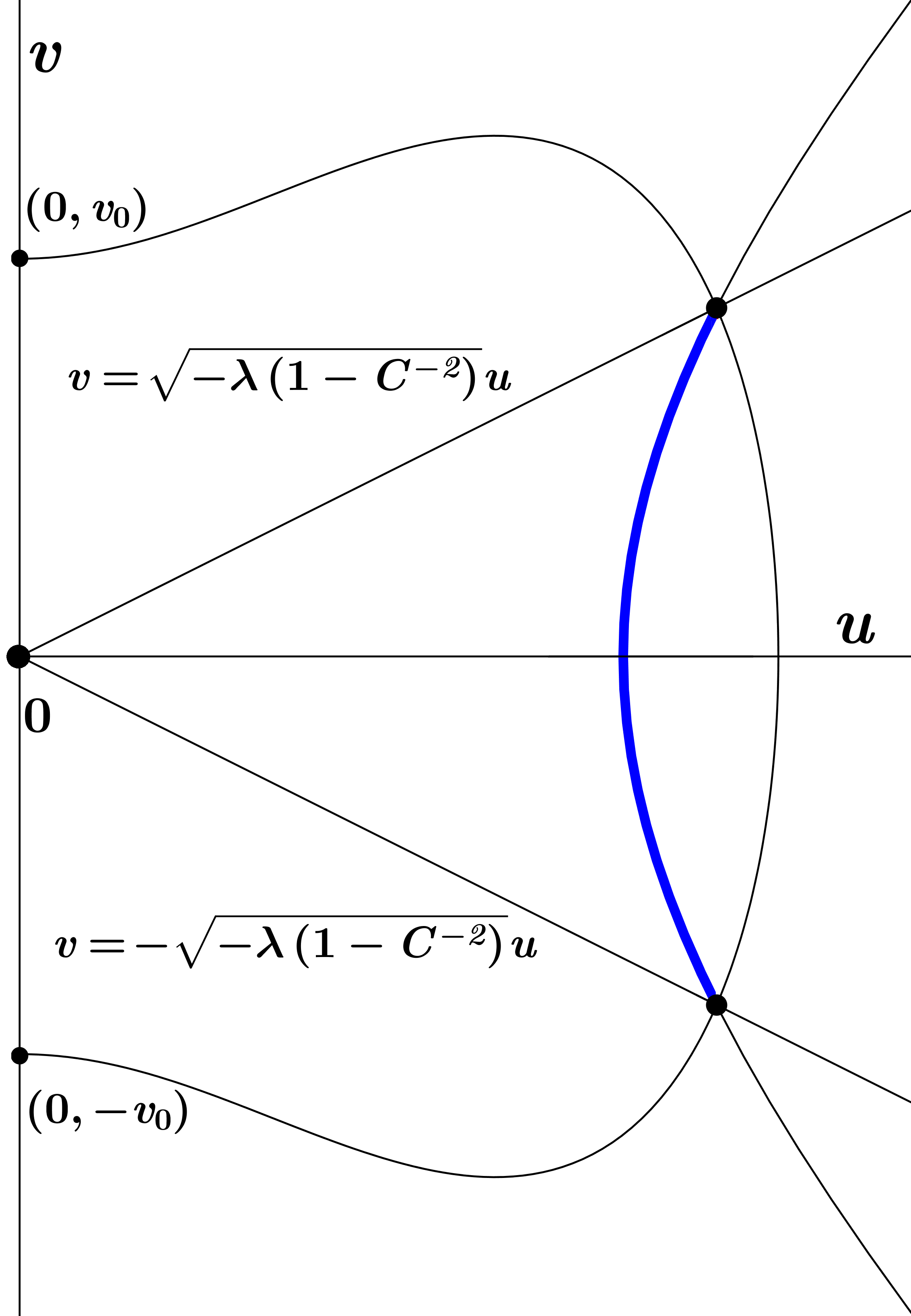}\;\;
	\includegraphics[scale=0.2]{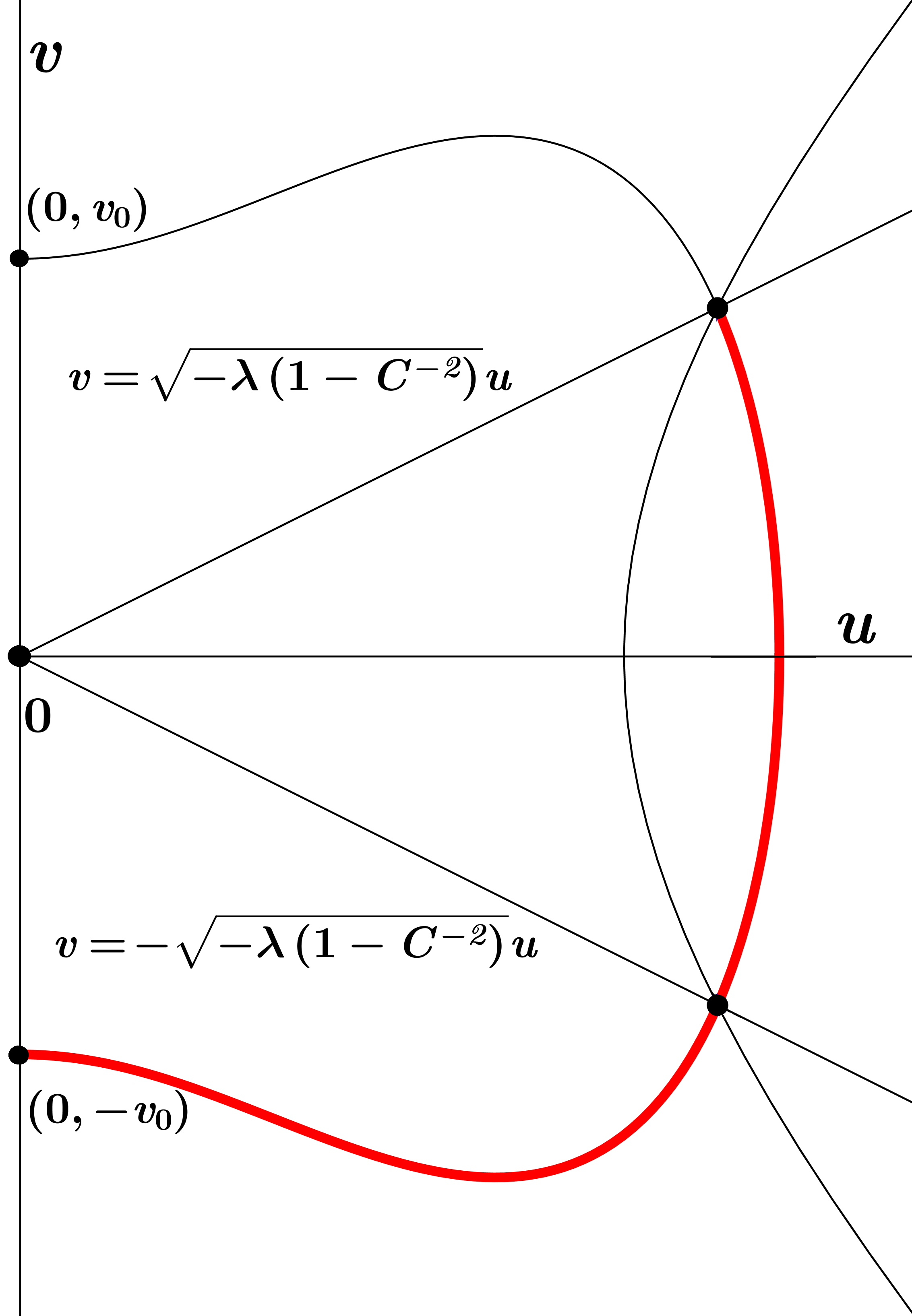}
	\caption{The phase plane route of a symmetric positive solution of \eqref{1.1} in the case $\l<0$: nonlinear interval $[0,\tfrac{1-h}{2}]$ (left), linear interval $(\tfrac{1-h}{2},\tfrac{1+h}{2})$ (center) and nonlinear interval $[\tfrac{1+h}{2},1]$ (right).}
	\label{Fig2}
\end{figure}

\noindent Indeed, such a solution leaves $(0,v_0)$ and follows the trajectory
$\mathscr{E}_{\mathscr{N}}=\frac{v_0^2}{2}$ until reaching $(u_\frac{1-h}{2},
v_\frac{1-h}{2})$ at $x=\frac{1-h}{2}$. Then, it follows the trajectory $\mathscr{E}_{\mathscr{L}}(u,v)=\frac{\l}{2}u_+^2$ of the linear system ($\mathscr{L}$)
until reaching $(u_\frac{1-h}{2},-v_\frac{1-h}{2})$ at $x= \frac{1+h}{2}$. Finally, it travels along the trajectory $\mathscr{E}_{\mathscr{N}}=\frac{v_0^2}{2}$
until reaching $(0,-v_0)$ at $x=1$, as sketched in the three plots of Figure \ref{Fig2}.
\par
\vspace{0.2cm}
\noindent \emph{Step 3:} Therefore, \eqref{1.1} has a symmetric positive solution as soon as there exists some $u_0>u_{\rm ho}$ such that
\begin{equation}
	\label{2.7}
	T_{\mathscr{N}}(\l,u_0) + {T}_{\mathscr{N}}(\l,u_0,u_\ell)=\frac{1-h}{2},
\end{equation}
where ${T}_{\mathscr{N}}(\l,u_0,u_\ell)$ denotes the necessary time to link $(u_0,0)$
with $(u_\ell,-v_\ell)$ along the trajectory \eqref{2.3} of ($\mathscr{N}$). According to \eqref{2.3}, it is easily seen that
\begin{equation}
	\label{2.8}
	T_{\mathscr{N}}(\l,u_0,u_\ell):=
	\int_\frac{u_\ell}{u_0}^1\frac{ds}{\sqrt{\l(1-s^2)+\frac{2}{p+1}u_0^{p-1}(1-s^{p+1})}}.
\end{equation}
Thus,
$$
\lim_{u_0\uparrow \infty}  T_{\mathscr{N}}(\l,u_0,u_\ell)=0.
$$
Consequently, we can infer from  \eqref{2.5} that there exists $u_0>u_{\rm ho}$ for which
\eqref{2.7} holds. So, for every $\l<0$, \eqref{1.1} admits, at least, one
symmetric positive solution. In Section 2.3, we will prove that, actually,
\eqref{1.1} has a unique symmetric positive solution. Let denote it by
$u\equiv u_\l$. According to the previous construction,
$$
\|u\|_\infty = u_0
$$
is reached at some points $x_0\in (0,\frac{1-h}{2})$ and
$x_1\in (\frac{1+h}{2},1)$ such that $x_0+x_1=1$. Moreover, by Theorem 2.1 of
Cubillos et al. \cite{CLGT-2024}, we already know that
\begin{equation}
	\label{2.9}
	\lim_{\l\downarrow -\infty}u_\l =0\;\; \hbox{uniformly on compact subsets of}\;\;
	(\tfrac{1-h}{2},\tfrac{1+h}{2}).
\end{equation}

\subsection{Symmetric solutions for $\l\geq 0$}
\label{sec2.2}
\noindent Throughout this section, we assume that $\l\geq0$. Then, the unique equilibrium of $(\mathscr N)$ is $(0,0)$, which is a half-right nonlinear center. Moreover, for every $u_0>0$, there exists $T:=T_{\mathscr{N}}(\l,u_0)>0$ such that the unique solution of ($\mathscr{N}$), $(u(t),v(t))$, such that $(u(0),v(0))=(u_0,0)$, is defined in $[-T,T]$, and it satisfies
\eqref{2.3} for all $t\in[-T,T]$. As in Section 2.1,
$$
(u(\pm T),v(\pm T))=(0,\mp v_0),\quad   v_0:=
\sqrt{\l u_0^2 +\tfrac{2}{p+1}u_0^{p+1}},
$$
and \eqref{2.4} holds. Thus, for every $\l\geq 0$,  $T_{\mathscr{N}}(\l,u_0)$ is decreasing with respect to $u_0$, though now, instead of \eqref{2.5}, we have that
\begin{equation}
	\label{2.10}
	\lim_{u_0\downarrow 0}{T}_{\mathscr{N}}(\l,u_0)=\frac{\pi}{2\sqrt{\l}}, \qquad
	\lim_{u_0\uparrow +\infty}{T}_{\mathscr{N}}(\l,u_0)=0,
\end{equation}
for all $\l>0$, and
\begin{equation}
	\label{2.11}
	\lim_{u_0\downarrow 0}{T}_{\mathscr{N}}(0,u_0)=\infty, \qquad
	\lim_{u_0\uparrow +\infty}{T}_{\mathscr{N}}(0,u_0)=0.
\end{equation}
\par
Regarding the linear system ($\mathscr{L}$), we must differentiate the cases $\l=0$ and $\l>0$. Indeed, when $\l=0$ all the energy lines are constant in $v$ and the $u$-axis of the phase plane consists of equilibria. Thus, in this case, a symmetric positive solution of \eqref{1.1} corresponds to a solution starting at some point $(0,v_0)$, with $v_0>0$, and running along the trajectory  $\mathscr{E}_{\mathscr{N}}=v_0^2/2$ until reaching some point $(u_0,0)$ at time $x=\frac{1-h}{2}$. Then, this solution should stay at the equilibrium $(u_0,0)$ for a time $h$ before completing the trajectory $\mathscr{E}_{\mathscr{N}}=v_0^2/2$ and reaching $(0,-v_0)$ at $x=1$. By the monotonicity of $T_{\mathscr{N}}(0,u_0)$ with respect to $u_0$, the existence and the uniqueness of the symmetric positive solution when $\l=0$ follows from \eqref{2.11}. These solutions are constant in the interval $(\frac{1-h}{2},\frac{1+h}{2})$.
\par
Next, we assume that $\l>0$. Then, instead of \eqref{2.2}, it is
more appropriate to consider the equivalent systems
\begin{equation}
	\label{2.12}
	(\mathscr{N}_\mathrm{eq})\;\; \left\{
	\begin{array}{ll}
		u'=\sqrt{\l}\,v,\\[3pt]
		v'=-\sqrt{\l}\,u-\frac{1}{\sqrt{\l}}\, u^{p},
	\end{array}
	\right.\qquad
	(\mathscr{L}_\mathrm{eq})\;\; \left\{
	\begin{array}{ll}
		u'=\sqrt{\l}\,v,\\[3pt]
		v'=-\sqrt{\l}\,u,
	\end{array}
	\right.
\end{equation}
whose associated total energies are,
$$
\mathscr{E}_{\mathscr{N}_{\rm eq}}(u,v):=\sqrt{\l}\left( \frac{1}{2}v^2+\frac{1}{2}u^2+\frac{1}{\l(p+1)}u^{p+1}\right)
$$
and
$$
\mathscr{E}_{\mathscr{L}_{\rm eq}}(u,v):=\sqrt{\l}\left(\frac{1}{2}v^2+ \frac{1}{2}u^2\right),
$$
respectively. For $(\mathscr{L}_\mathrm{eq})$ in \eqref{2.12}, the angular variable in polar coordinates, $\t(x)$, satisfies
$$
\t'(x)=\frac{v'(x)u(x)-u'(x)v(x)}{u^2(x)+v^2(x)}=-\sqrt{\l}
$$
and hence, the angle run, clockwise,  by a periodic solution in an interval of length $h$ is
$$
\t=\int_0^h\t'(x)\,dx = -h\sqrt{\l}.
$$
Thus, when $\l>0$, a symmetric positive solution of \eqref{1.1} should begin at some point $(0,v_0)$ and reach at time $x=(1-h)/2$, through the level set
\begin{equation}
	\label{ii.13}
	\mathscr{E}_{\mathscr{N}_{\rm eq}}(u,v)=\sqrt{\l}\frac{v_0^2}{2},
\end{equation}
the (unique) point $(u_\ell,v_\ell)$ lying on the line
\begin{equation}
	\label{ii.14}
	v=\Big(\!\!\tan \tfrac{h\sqrt{\l}}{2}\Big)u.
\end{equation}
Then, it moves in a time $h$ an angle $-h\sqrt{\l}$ trough the level set
$$
\mathscr{E}_{\mathscr{L}_{\rm eq}}=\sqrt{\l}\left(\frac{1}{2}v_{\ell}^2+\frac{1}{2}u_{\ell}^2\right)
$$
up to reach the point  $(u_{\ell},-v_{\ell})$. Finally, it runs again along \eqref{ii.13} from $(u_{\ell},-v_{\ell})$ to $(0,-v_0)$, as it has been sketched in Figure \ref{Fig3}, where the phase plane route in the nonlinear (resp. linear) level set has been highlighted using red (resp. green) color.
Therefore, to prove the existence of a symmetric positive solution it suffices to show that $(0,v_0)$ can be linked in time $(1-h)/2$ to a point $(u_\ell,v_\ell)$ lying on the line \eqref{ii.14} through the integral curve
$$
v^2+u^2+\frac{2}{\l(p+1)}u^{p+1}=\Big(\!\!\tan \tfrac{h\sqrt{\l}}{2}\Big)^2u_\ell^2+u_\ell^2+\frac{2}{\l(p+1)}u_\ell^{p+1}.
$$
\begin{figure}[h!]
	\centering
	\includegraphics[scale=0.2]{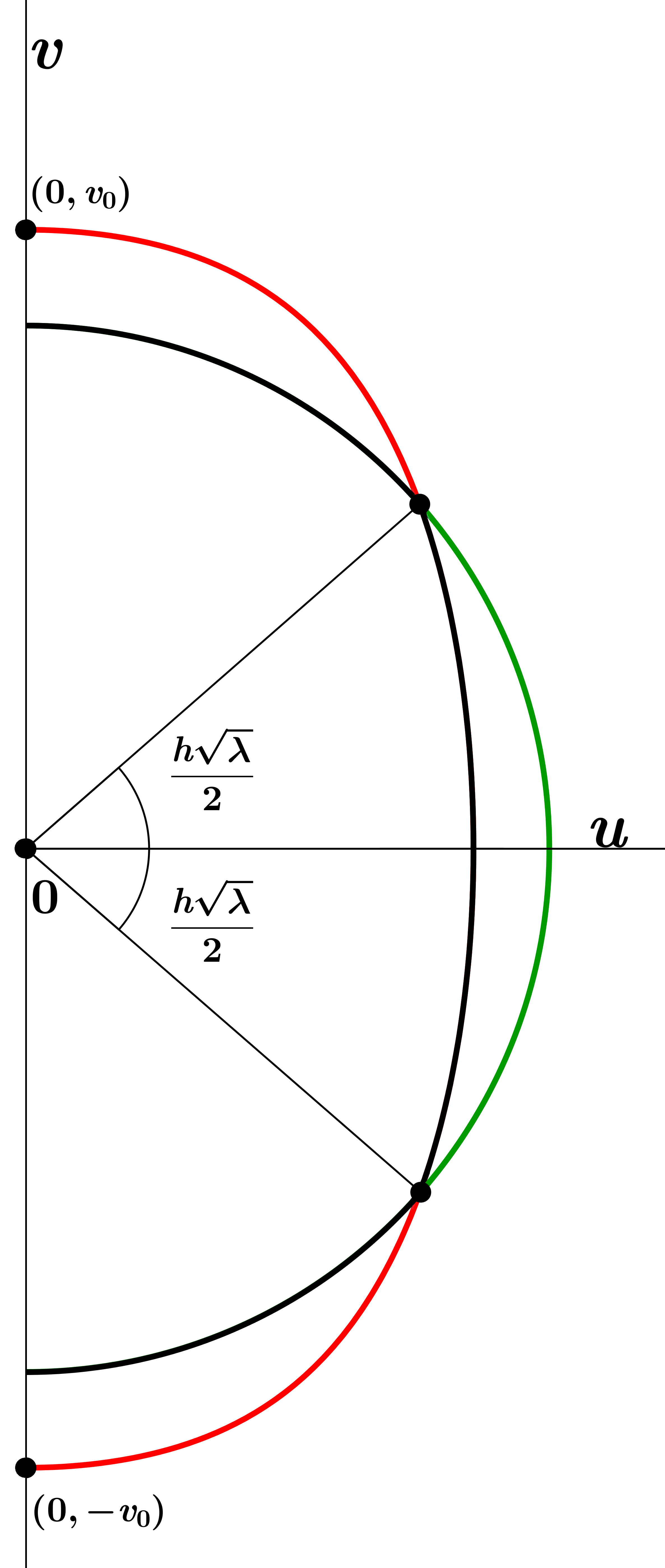}
	\caption{A symmetric positive solution of \eqref{1.1} in the case $\l\in(0,\pi^2)$.}
	\label{Fig3}
\end{figure}
Note that the necessary time to link $(0,v_0)$ and $(u_\ell,v_\ell)$ along this integral curve
is given by
\begin{equation}
	\label{ii.15}
	T_{\mathscr{N}}(\l,0,u_\ell):=\frac{1}{\sqrt{\l}}\int_0^1\frac{ds}{\sqrt{\tan^2 \tfrac{h\sqrt{\l}}{2}+1-s^2+\frac{2u_\ell^{p-1}}{\l(p+1)}(1-s^{p+1})}},
\end{equation}
where we have used that $u'=\sqrt{\l}v$. Thus, since $T_{\mathscr{N}}(\l,0,u_\ell)$ is decreasing with respect to $u_\ell$ and
$$
\lim_{u_\ell\uparrow +\infty}{T}_{\mathscr{N}}(\l,0,u_\ell)=0,
$$
the existence, and the uniqueness, of a symmetric positive solution of \eqref{1.1} in case $\l>0$ is guaranteed as soon as
\begin{equation}
	\label{ii.16}
	\lim_{u_\ell\downarrow 0}{T}_{\mathscr{N}}(\l,0,u_\ell)= \frac{1}{2} \Big( \frac{\pi}{\sqrt{\l}}-h\Big)>\frac{1-h}{2},
\end{equation}
which holds true if, and only if, $\l<\pi^2$. To show the validity of the limit in \eqref{ii.16}, we set
$$
\a:=\tan^2 \tfrac{h\sqrt{\l}}{2}+1=\sec^2 \tfrac{h\sqrt{\l}}{2}.
$$
Then,
\begin{equation*}
	\begin{split}
		\lim_{u_\ell\downarrow 0}{T}_{\mathscr{N}}(\l,0,u_\ell)&=\frac{1}{\sqrt{\l}}\int_0^1\frac{ds}{\sqrt{\tan^2\tfrac{h\sqrt{\l}}{2}
				+1-s^2}}=\frac{1}{\sqrt{\l}}\int_0^{\frac{1}{\sqrt{\a}}}\frac{d\t}{\sqrt{1-\t^2}}\\&=
		-\frac{1}{\sqrt{\l}}\Big( \arccos \cos  \tfrac{h\sqrt{\l}}{2}-\frac{\pi}{2}\Big) = \frac{1}{2}\Big( \frac{\pi}{\sqrt{\l}}-h\Big),
	\end{split}
\end{equation*}
which provides us with \eqref{ii.16}. Summarizing, by the monotonicity of $T_{\mathscr{N}}(\l,0,u_\ell)$ with respect to $u_\ell$, for every $\l\in(0,\pi^2)$, \eqref{1.1} has a unique symmetric positive solution.
\par
Finally, note that condition $\l<\pi^2$ is actually necessary for the existence of a positive solution of
\eqref{1.1}. Indeed, if \eqref{1.1} has a positive solution, $u$, then multiplying the differential equation
by $\sin (\pi x)$ and integrating by parts in $(0,1)$ yields
\begin{align*}
	\pi^2\int_0^1\sin(\pi x)u(x)\,dx & =\l \int_0^1\sin(\pi x)u(x)\,dx+ \int_0^1 a(x)\sin(\pi x)u^{p}(x)\,dx\\ &   > \l \int_0^1\sin(\pi x)u(x)\,dx
\end{align*}
Thus, $\l<\pi^2$, because $u(x)>0$ for all $x\in (0,1)$ and $a\gneq 0$.

%
%and hence, by the uniqueness of the principal eigenvalue,
%$$
%  \l = \s[-D^2-a(x)u^2;\mathscr{D}],
%$$
%where $\s[-D^2-a(x)u^2;\mathscr{D}]$ stands for the principal eigenvalue of $-D^2-au^2$ under Dirichlet
%boundary conditions on $0$ and $1$. Thus, since $u(x)>0$ for all $x\in (0,1)$ and $a\gneq 0$, by the monotonicity properties of the principal eigenvalues (see \cite{LG-2013}, if necessary), it is apparent that $\l<\pi^2$.

\subsection{Uniqueness of the symmetric positive solution}

\noindent The next result establishes the uniqueness of the symmetric positive solutions of \eqref{1.1}. Note that, in case $\l\in[0,\pi^2)$, this result  has been already established in Section \ref{sec2.2}

\begin{theorem}
	\label{th2.1}
	For every  $\l\in (-\infty,\pi^2)$, the problem \eqref{1.1} has a unique symmetric positive solution.
	Thus, it has a positive solution if, and only if, $\l<\pi^2$.
\end{theorem}

\begin{proof}
	It suffices to prove the result for $\l<0$. So, suppose $\l<0$. Arguing by contradiction, assume that $u_1$ and $u_2$ are two symmetric positive solutions of \eqref{1.1} and set
	$$
	v_1=u_1',\quad v_2=u_2'.
	$$
	Since $u_1(0)=u_2(0)=0$, by the uniqueness of the underlying Cauchy problem,
	$v_1(0)\neq v_2(0)$. By inter-exchanging their names, if necessary, we can assume that $v_1(0)>v_2(0)$. Then,
	$$
	w:= u_1-u_2
	$$
	satisfies $w(0)=0$ and $w'(0)>0$. Thus, there exists $\eta\in (0,1]$ such that $w(x)>0$ for all $x\in (0,\eta)$, and $w(\eta)=0$. We claim that $\eta=1$, i.e.
	\begin{equation}
		\label{ii.17}
		w(x)>0\quad \hbox{for all}\;\; x\in (0,1).
	\end{equation}
	Indeed, since \eqref{1.1} is autonomous in $[0,\frac{1-h}{2}]$, $u_1$ and $u_2$  lie in different energy levels there in and, in particular, $v_1(0)>v_2(0)>0$. Thus, since
	$$
	\mathscr{E}_{\mathscr{N}}(0,v_1(0))>\mathscr{E}_{\mathscr{N}}(0,v_2(0)),
	$$
	it becomes apparent that, for every $x\in [0,\frac{1-h}{2}]$,
	\begin{equation}
		\label{ii.18}
		\mathscr{E}_{\mathscr{N}}(u_1(x),v_1(x))\!=\!\mathscr{E}_{\mathscr{N}}(0,v_1(0))  \!>\!\mathscr{E}_{\mathscr{N}}(u_2(x),v_2(x))\!=\!\mathscr{E}_{\mathscr{N}}(0,v_2(0))\!>\!0.
	\end{equation}
	Moreover, as sketched in Figure \ref{Fig2}, since $u_1$ and $u_2$ are symmetric positive solutions, $u_1(\frac{1-h}{2})$ and $u_2(\frac{1-h}{2})$ should lie in the straight line
	$$
	v=-\sqrt{-\l(1-C^{-2})}u.
	$$
	Consequently, substituting in \eqref{ii.18}, we find that
	\begin{equation}
		\label{ii.19}
		\begin{split}
			 \mathscr{E}_{\mathscr{N}}(u_1,v_1)&=\frac{-\l(1-C^{-2})}{2}u_1^2(\tfrac{1-h}{2})+\frac{\l}{2}u_1^2(\tfrac{1-h}{2})
			 +\frac{1}{p+1}u_1^{p+1}(\tfrac{1-h}{2})\\&=\frac{\l}{2C^2}u_1^2(\tfrac{1-h}{2})+\frac{1}{p+1}u_1^{p+1}
			(\tfrac{1-h}{2})\\&>\mathscr{E}_{\mathscr{N}}(u_2,v_2)=\frac{\l}{2C^2}u_2^2(\tfrac{1-h}{2})
			+\frac{1}{p+1}u_2^{p+1}(\tfrac{1-h}{2})>0.
		\end{split}
	\end{equation}
	Thus, since the function
	$$
	\varphi(s):=\frac{\l s^2}{2C^2}+\frac{s^{p+1}}{p+1},\qquad s>0,
	$$
	is increasing if $\v(s)>0$, because, in this case,  $\varphi'(s)=\frac{\l s}{C^2}+s^p>0$ for all $s>0$, \eqref{ii.19} implies that
	$$
	w(\tfrac{1-h}{2})=u_1(\tfrac{1-h}{2})-u_2(\tfrac{1-h}{2})>0.
	$$
	In particular, $\eta\neq \frac{1-h}{2}$. Suppose that $\eta \in (0,\frac{1-h}{2})$. Then, since
	since $w(x)>0$ for all $x\in(0,\eta)$ and $w(\eta)=0$, necessarily, $w'(\eta)<0$. Moreover,
	since $w(\frac{1-h}{2})>0$, there exists $\d\in(\eta,\frac{1-h}{2})$ such that
	$$
	w(x)<0\;\;\hbox{for all}\;\; x\in(\eta,\d),\;\;w(\d)=0,\;\;\hbox{and}\;\; w'(\d)>0.
	$$
	On the other hand, by \eqref{ii.18}, we already know that, for every $x\in[0,\tfrac{1-h}{2}]$,
	\begin{equation}
		\label{ii.20}
		\begin{split}
			\mathscr{E}_{\mathscr{N}}& (u_1(x),v_1(x))=\frac{1}{2}v_1^2(x)+\frac{\l}{2}u_1^2(x)+\frac{1}{p+1}u_1^{p+1}(x)\\&>\mathscr{E}_{\mathscr{N}}(u_2(x),v_2(x))=\frac{1}{2}v_2^2(x)+\frac{\l}{2}u_2^2(x)+\frac{1}{p+1}u_2^{p+1}(x)>0.
		\end{split}
	\end{equation}
	Consequently, since
	$$
	w(\eta)=u_1(\eta)-u_2(\eta) = w(\d)=u_1(\d)-u_2(\d)=0,
	$$
	it follows from \eqref{ii.20} that
	$$
	v_1^2(\eta)>v_2^2(\eta),\quad v_1^2(\d)>v_2^2(\d),
	$$
	or, equivalently,
	$$
	|v_1(\eta)|>|v_2(\eta)|,\quad |v_1(\d)|>|v_2(\d)|.
	$$
	Moreover, since $w'(\eta)<0$, we have that
	\[
	0>w'(\eta)=v_1(\eta)-v_2(\eta)\;\; \hbox{and}\;\; |v_1(\eta)|>|v_2(\eta)|.
	\]
	So, $v_1(\eta)<0$. By the clockwise rotation sense in the phase plane, this entails  that $v_1(x)<0$ for all $x\in[\eta,\tfrac{1-h}{2}]$. Similarly, since
	\[
	0<w'(\d)=v_1(\d)-v_2(\d)\;\; \hbox{and}\;\; |v_1(\d)|>|v_2(\d)|,
	\]
	we can infer that $v_1(\d)>0$, which is a contradiction. Therefore,
	$$
	w(x)=u_1(x)-u_2(x)>0\;\; \hbox{for all}\;\; x\in(0,\tfrac{1-h}{2}].
	$$
	By the symmetry of the solution, we also have that
	$$
	w(x)=u_1(x)-u_2(x)>0\;\; \hbox{for all}\;\; x\in[\tfrac{1+h}{2},1).
	$$
	Finally, since $w(\tfrac{1-h}{2})>0$ and $w(\tfrac{1+h}{2})>0$, if $w$ changes sign in $(\tfrac{1-h}{2},\tfrac{1+h}{2})$, then there would exist a subinterval $[\a,\b]\subsetneq(\tfrac{1-h}{2},\tfrac{1+h}{2})$ such that $w(x)<0$ for all $x\in (\a,\b)$, and
	\begin{equation}
		\label{ii.21}
		\left\{
		\begin{array}{l}
			-w''=\l w\qquad \hbox{in}\; [\a,\b],\\[1pt]
			w(\a)=0=w(\b).
		\end{array}
		\right.
	\end{equation}
	Consequently,  $w_{|[\a,\b]}$ is an eigenfunction with associated eigenvalue
	$$
	\l=\pi^2/(\b-\a)^2>0,
	$$
	which contradicts our assumption that $\l<0$.  Therefore, \eqref{ii.17} holds.
	\par
	Subsequently, we fix  $\l<0$ and denote by $D^2$ the differential operator $\frac{d^2}{dx^2}$.
	Then, $w:=u_1-u_2$ satisfies
	\begin{equation}
		\label{ii.22}
		\left\{
		\begin{array}{ll}
			-D^2w-a(x)(u_1^p-u_2^p)=\l w,\quad x\in [0,1],\\[2pt]
			w(0)=0=w(1).
		\end{array}
		\right.
	\end{equation}
	Setting
	$$
	\psi(\xi):=(\xi u_1+(1-\xi)u_2)^p,\qquad \xi\geq 0,
	$$
	it is apparent that
	\begin{align*}
		u_1^p-u_2^p& =\psi(1)-\psi(0)=\int_0^1\psi'(\xi)\,d\xi
		\\ & =p\int_0^1(\xi u_1+(1-\xi)u_2)^{p-1}\, d\xi\,(u_1-u_2).
	\end{align*}
	Thus, the problem \eqref{ii.22} can be equivalently written as
	\begin{equation}
		\label{ii.23}
		\left\{
		\begin{array}{ll}
			\Big(-D^2-a(x)p\int_0^1(\xi u_1+(1-\xi)u_2)^{p-1}\, d\xi\Big)w=\l w,\quad x\in(0,1),\\[6pt]
			w(0)=0=w(1).
		\end{array}
		\right.
	\end{equation}
	Thus, since, thanks to \eqref{ii.17},  $w>0$ in $(0,1)$, it becomes apparent that
	\begin{equation}
		\label{ii.24}
		\l=\s_1\Big[-D^2-a(x)p\int_0^1(\xi u_1+(1-\xi)u_2)^{p-1}\, d\xi\Big],
	\end{equation}
	where $\s_1[-D^2+b(x)]$ denotes the principal eigenvalue of $-D^2+b(x)$ under
	Dirichlet boundary conditions. Hence, by the monotonicity of the principal eigenvalue with respect to the potential (see, if necessary, \cite{LG-2013}),  it follows from \eqref{ii.24} that
	\[
	\l <\s_1[-D^2-a(x)u_1^{p-1}]=\l,
	\]
	which is impossible. Therefore,  $w\equiv0$, which ends the proof.
\end{proof}

\subsection{Structure of the symmetric positive solutions}

\noindent According to, e.g., Amann and L\'opez-G\'omez \cite{ALG},
or adapting the arguments of Lemma 2.3 and Theorem 2.4 of the authors in \cite{LGMHZ-2023}, it becomes apparent that the set of non-trivial solutions of the problem
\begin{equation}
	\label{ii.25}
	\left\{
	\begin{array}{ll}
		-u''=\l u+a(x)u^p,\quad x\in(0,\frac{1}{2}),\\[3pt]
		u(0)=u'(\frac{1}{2})=0,
	\end{array}
	\right.
\end{equation}
contains a component of positive solutions bifurcating from $\l=\pi^2$, denoted by
$\mathscr{C}^+_{0.5}$, such that
$$
\mc{P}_\l ( \mathscr{C}^+_{0.5} ) =(-\infty,\pi^2].
$$
As there is a canonical one-one correspondence between the positive solutions of \eqref{ii.25} and the symmetric positive solutions of \eqref{1.1}, we can infer that there exists a component of symmetric positive solutions of \eqref{1.1},  denoted by $\mathscr{C}_{h,s}^+$, such that
\begin{equation}
	\label{ii.26}
	(\pi^2,0)\in \mathscr{C}_{h,s}^+\;\;\hbox{and}\;\; \mc{P}_\l (\mathscr{C}_{h,s}^+)=(-\infty,\pi^2].
\end{equation}
Thus, thanks to Theorem \ref{th2.1}, the next result holds.

\begin{theorem}
	\label{th2.2}
	Besides \eqref{ii.26}, the component $\mathscr{C}_{h,s}^+$ satisfies the following properties:
	\begin{enumerate}
		\item[{\rm (a)}] For every $(\l,u_\l)\in\mathscr{C}_{h,s}^+$, either $\l=\pi^2$ and $u_\l=0$, or $\l<\pi^2$ and $u_\l$ is the unique symmetric positive solution of \eqref{1.1}.
		
		\item[{\rm (b)}] Subsequently, for every     $\l<\pi^2$, we denote by $u_\l$ the unique symmetric positive solution of \eqref{1.1}. Then, for every $\s \in(-\infty,\pi^2]$,
		\begin{equation}
			\label{ii.27}
			\lim_{\l\to \s}\|u_\l\|_{\infty}=0\;\;\hbox{if and only if} \;\; \s=\pi^2.
		\end{equation}
		
		\item[{\rm (c)}] For each $\l_0\in (-\infty,\pi^2)$, the set of solutions
		$$
		\mathscr{D}_0:=\{u_\l :\l \in [\l_0,\pi^2)\}
		$$
		is bounded in $\mc{C}^1([0,1];\R)$, though
		\begin{equation}
			\label{ii.28}
			\lim_{\l\downarrow -\infty }\|u_\l\|_{\infty}=+\infty.
		\end{equation}
		
		\item[{\rm (d)}] The map $\l\mapsto u_\l$ is continuous. Therefore,
		$$
		\mathscr{C}_{h,s}^+=\{(\l,u_\l)\;:\; \l \in (-\infty,\pi^2]\}
		$$
		is a continuous curve, as illustrated in Figure \ref{Fig4}.
	\end{enumerate}
	Finally, by the uniqueness of the component of positive solutions bifurcating from
	$u=0$ at $\l=\pi^2$, $\mathscr{C}_{h,s}^+$ must be a subcomponent of the component of
	positive solutions $\mathscr{C}_{h}^+$ whose existence was established in Section \ref{sec1}.
\end{theorem}

\begin{proof}
	Part (a) is a byproduct of the previous analysis. Part (b) is a consequence of the fact that
	$\mathscr{C}^+_{0.5}$ bifurcates from $u=0$ at $\l=\pi^2$. Thus, by construction,
	so does it the component $\mathscr{C}_{h,s}^+$. Moreover, it is well known that $\l=\pi^2$ is the unique bifurcation value to positive solutions from $u=0$. Therefore, \eqref{ii.27} holds.
	\par
	The fact that, for every $\l_0 <\pi^2$, the set $\mathscr{D}_0$ is bounded in $\mc{C}^1([0,1];\R)$
	is a direct consequence, e.g.,  of the a priori bounds of \cite{ALG} and \cite{LGMHZ-2023}. To show \eqref{ii.28}, suppose that $\l<0$ and let $x_m=x_m(\l)\in (0,1)$ be such that
	$$
	u_\l(x_m)=\|u_\l\|_\infty.
	$$
	Then, $u'_\l(x_m)=0$ and $u''_\l(x_m)\leq 0$. Thus, it follows from \eqref{1.1} that
	$$
	\l + a(x_m) u_\l^{p-1}(x_m) \geq 0
	$$
	and hence, $a(x_m)=1$. Therefore,
	$$
	u_\l(x_m)=\|u_\l\|_\infty \geq (-\l)^\frac{1}{p-1}
	$$
	and, letting $\l\downarrow -\infty$ in this estimate, \eqref{ii.28} holds. This shows Part (c).
	\par
	The continuity of the map $\l\mapsto u_\l$ follows, through a rather standard compactness argument,
	from the uniqueness of the symmetric positive solution. Indeed, necessarily,
	$$
	\lim_{\l\to \mu}||u_\l-u_\mu\|_\infty =0
	$$
	for all $\mu \in (-\infty,\pi^2]$. Note that Part (c) guarantees the existence of a priori bounds
	in $[\mu-1,\pi^2)$. This ends the proof.
\end{proof}

Figure \ref{Fig4} shows two admissible global bifurcation diagrams of positive solutions of
\eqref{1.1}. In the first bifurcation diagram we are assuming that $\mathscr{C}^+_{h}$ consists of
radially symmetric positive solutions and, hence, $\mathscr{C}^+_{h} = \mathscr{C}^+_{h,s}$. In the second one
we are representing a case when $\mathscr{C}^+_{h,s}$ is a proper subset of $\mathscr{C}^+_{h}$.
One of the main goals of this paper is ascertaining whether, or not, each of these situations can occur.

\begin{figure}[h!]
	\centering
	\includegraphics[scale=0.51]{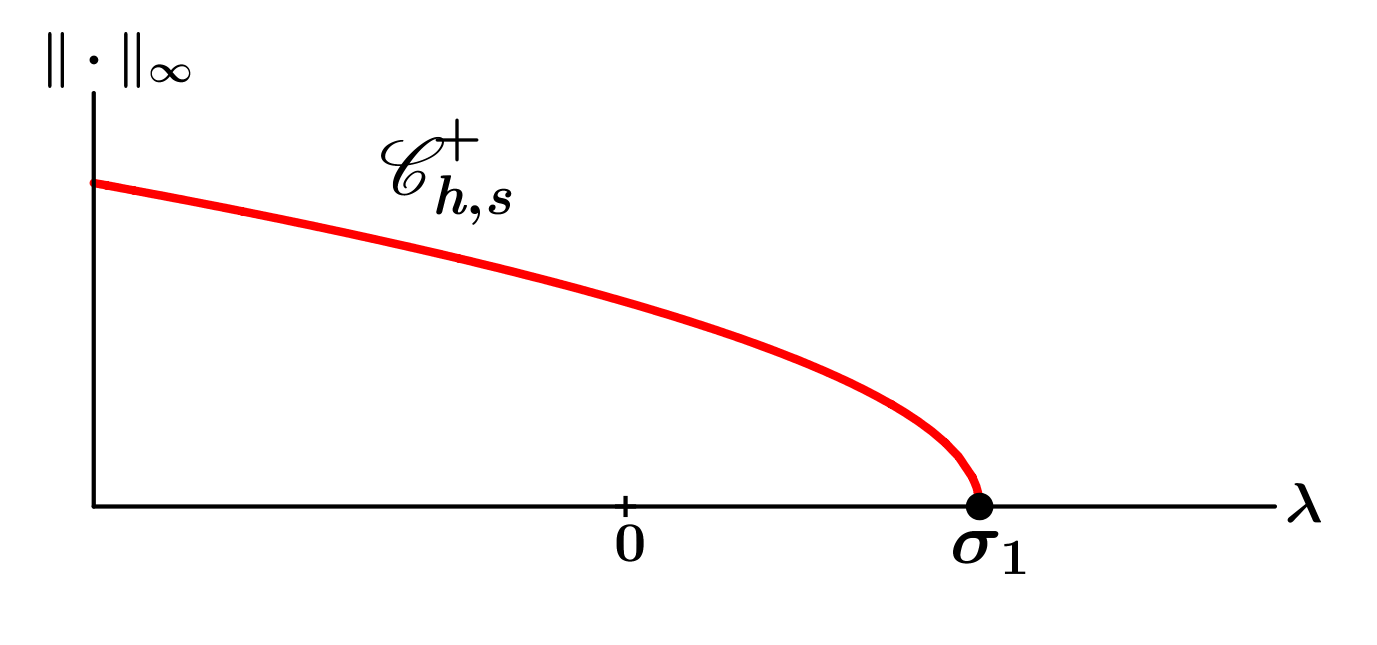}
	\includegraphics[scale=0.51]{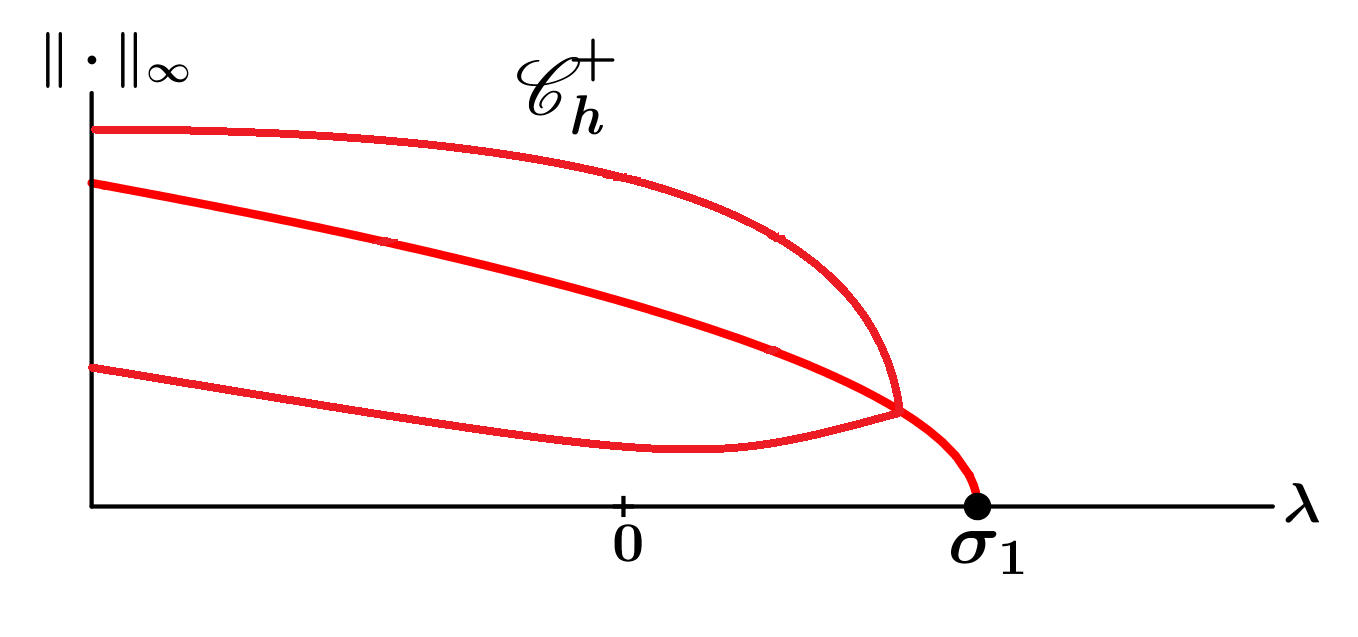}
	\caption{The component $\mathscr{C}_{h,s}^+$ of \eqref{1.1} positive symmetric solutions (left) and a possible component $\mathscr{C}_h^+$ of \eqref{1.1} positive solutions (right).}
	\label{Fig4}
\end{figure}

\setcounter{equation}{0}

\section{Limiting behavior of $\mathscr{C}^+_{h}$ as $h\uparrow1$}

\noindent The main goal of this section is analyzing the behavior of the component of positive solutions of \eqref{1.1}, assuming \eqref{1.3}, bifurcating from $u=0$ at $\l=\pi^2$, $\mathscr{C}^+_{h}$, as $h\uparrow1$, i.e. when the nonlinear system ($\mathscr{N}$) tends towards the linear system ($\mathscr{L}$). In the limiting situation when $h=1$, \eqref{1.1} reduces to ($\mathscr{L}$), which admits some positive solution if, and only if, $\l=\pi^2$. Actually, the component of positive solutions bifurcating from $u=0$ at $\l=\pi^2$ for this limiting problem, denoted by $\mathscr{C}_1^+$, is given by
$$
   \mathscr{C}_1^+ =\{(\pi^2,\a \sin(\pi x))\;:\; \a\in \R_+:=[0,+\infty)\}.
$$
Thus, $\l=\pi^2$ is a bifurcation value to a vertical bifurcation from $u=0$. In particular,
$\mc{P}_\l(\mathscr{C}_1^+)=\{\pi^2\}$. Consequently, the remaining positive solutions of the component
$\mathscr{C}_{h}^+$ for $\l<\pi^2$ should  be unbounded as $h\uparrow 1$. The next result provides us with the point-wise behavior of the the positive solutions of \eqref{1.1} as $h \uparrow 1$.
\par

\begin{theorem}
\label{th3.1}
For any given $\l<\pi^2$ and, for  every $h\in [0,1)$, let $u_h$ be a positive solution of
\eqref{1.1} for the choice \eqref{1.3}. Then,
\begin{equation}
\label{3.1}
	\lim_{h\uparrow 1}\| u_h\|_{\infty}=+\infty.
\end{equation}
Moreover,
\begin{equation}
\label{3.2}
	\lim_{h\uparrow 1}u_h(x)=+\infty \;\; \hbox{for all}\;\; x\in (0,1).
\end{equation}
In particular, the components $\mathscr{C}_h^+$ blow-up, point-wise in $(0,1)$, as $h\uparrow 1$.
\end{theorem}
\begin{proof}
As in \cite[Lemma 2.1]{FLG-2022} or \cite[Lemma 2.2]{LGMHZ-2023}, we denote by $\mc{K}:\mc{C}([0,1];\R)\to \mc{C}^1([0,1];\R)$ the linear, continuous and compact operator defined, for every $f\in \mc{C}([0,1];\R)$, by
$$
  \mc{K}f(x):=\int_0^x (y-x)f(y)\,dy-x\int_0^1 (y-1)f(y)\,dy,
$$
which provides us with the unique solution of the linear problem
\begin{equation*}
\left\{
\begin{array}{l}
-u'' = f \quad \hbox{in}\;\; [0,1],\\[1pt]
u(0)=u(1)=0.
\end{array}
\right.
\end{equation*}
Thus,
\begin{equation}
\label{3.3}
	u_h=\mc{K}(\l u_h+a_hu_h^p)\;\;\hbox{for all}\;\; h\in [0,1).
\end{equation}
Arguing by contradiction, suppose that there is  a sequence, $\{h_n\}_{n\geq 1}$ in $(0,1)$ converging to 1, such that, for some constant $C>0$,
\begin{equation}
\label{3.4}
  \|u_{h_n}\|_\infty \leq C \;\; \hbox{for all}\;\; n \geq 1.
\end{equation}
Then, \eqref{3.3} implies that, for every $n\geq 1$,
\begin{equation}
\label{3.5}
 \frac{u_{h_n}}{\|u_{h_n}\|_{\infty}}=\mc{K}\left(\l\frac{u_{h_n}}{\|u_{h_n}\|_{\infty}}+
 a_{h_n}u_{h_n}^{p-1}\frac{u_{h_n}}{\|u_{h_n}\|_{\infty}}\right).
\end{equation}
Moreover, thanks to \eqref{3.4}, the sequence
$$
 \l\frac{u_{h_n}}{\|u_{h_n}\|_{\infty}}+a_{h_n}(x)u_{h_n}^{p-1}\frac{u_{h_n}}{\|u_{h_n}\|_{\infty}},
 \quad n\geq 1,
$$
is bounded in $\mc{C}([0,1];\R)$. Thus, by the compactness of $\mc{K}$, there exists a subsequence of
$\{u_{h_n}\}_{n\geq 1}$, labeled by $h_{n_m}$, $m\geq 1$, such that
$$
   \lim_{m\to+\infty}\frac{u_{h_{n_m}}}{\|u_{h_{n_m}}\|_\infty}= \psi
$$
in $\mc{C}^1([0,1];\R)$, for some   $\psi \in \mc{C}^1([0,1];\R)$ such that
$\|\psi\|_{\infty}=1$ and $\psi\gneq 0$. Consequently, particularizing \eqref{3.5} at
$h_{n_m}$ and letting $m\to \infty$, shows that
$$
\left\{
	\begin{array}{l}
		-\psi''=\l\psi\quad\hbox{in}\;(0,1),\\[1pt]
		\psi(0)=\psi(1)=0,
	\end{array}
	\right.
$$
which is impossible because $\psi\gneq 0$ and $\l<\pi^2$. Therefore, \eqref{3.1} holds.
\par
To prove \eqref{3.2} we will use a concavity argument if $\l\geq0$ and a phase plane argument together with the analysis of the linear problem if $\l<0$.
\par
\vspace{0.2cm}
\noindent \emph{Case 1:} Assume $\l\geq0$. Then, for every $x\in[0,1]$,
$$
-u_{h}''(x)=\l u_{h}(x)+a_{h}(x)u_{h}^{p}(x)\geq0
$$
and, hence, $u_{h}$ is concave in $[0,1]$ for all $h\in[0,1]$. Let  $x_{m_h}\in (0,1)$ be such that
$$
    u_{h}(x_{m_h})=\|u_{h}\|_{\infty}.
$$
Then, for every $x\in(0,1)$, we have that
\begin{equation*}
x=\left\{
\begin{array}{ll}
\left(1-\frac{x}{x_{m_h}}\right)0+\frac{x}{x_{m_h}}x_{m_h}&\quad\hbox{if}\;\; x\leq x_{m_h},\\[6pt]
\left(1-\frac{1-x}{1-x_{m_h}}\right)1+\frac{1-x}{1-x_{m_h}}x_{m_h}&\quad\hbox{if}\;\; x > x_{m_h},
\end{array}
\right.
\end{equation*}
with $\frac{x}{x_{m_h}}\in(x,1]$ if $x\leq x_{m_h}$, and $
\frac{1-x}{1-x_{m_h}}\in(1-x,1)$ if $x>x_{m_h}
$. Consequently, by the concavity of $u_{h}$,
\begin{equation*}
u_{h}(x)\geq
\left\{
\begin{array}{ll}
\left(1-\frac{x}{x_{m_h}}\right)u_{h}(0)+\frac{x}{x_{m_h}}\|u_{h}\|_{\infty}\geq x\|u_{h}\|_{\infty}&\hbox{if}\;\; x\leq x_{m_h},\\[6pt]
\left(1-\frac{1-x}{1-x_{m_h}}\right)u_{h}(1)+\frac{1-x}{1-x_{m_h}}
\|u_{h}\|_{\infty}>(1-x)\|u_{h}\|_{\infty}&\hbox{if}\;\; x> x_{m_h}.
\end{array}
\right.
\end{equation*}
Therefore, by \eqref{3.1}, we find that, for every $x\in(0,1)$,
\[
\lim_{h\uparrow 1}u_{h}(x)=+\infty.
\]
This ends the proof in this case.
\par
\vspace{0.2cm}
\noindent \emph{Case 2:} Assume $\l<0$. We begin by constructing the geometric set of points, $\mathscr{A}_h$,  lying in the fourth (resp. first) quadrant and taking time $h$ to reach the negative (resp. positive) $v$-axis trough the linear flow of $(\mathscr{L})$. Note that,  for every positive solution $u$ of \eqref{1.1}, we have that
\begin{equation}
\label{c.6}
  (u(t),v(t))= (u(t),u'(t))\notin \mathscr{A}_h\;\;\hbox{for all}\;\; t \in [0,\tfrac{1-h}{2}],
\end{equation}
because $u(t)>0$ for all $t\in (0,1)$. To construct $\mathscr{A}_h$ we adapt an argument of Section \ref{sec2.1}. For every $v_-<0$, the required time to reach $(0,v_-)$ from a generic point
$$
(u_r,v_r)\in\{(u,v)\,:\,\mathscr{E}_{\mathscr{L}}(u,v)=v_-^2\}
$$
is given by
\begin{equation}
\label{c.7}
\begin{split}
{T}_{\mathscr{L}}(\l,u_r,0) & :=-\int_{u_r}^{0}\frac{du}{\sqrt{v_-^2-\l u^2}} =- \frac{1}{\sqrt{|\l|}}\int_{\frac{\sqrt{|\l|}u_r}{v_-}}^{0} \frac{ds}{\sqrt{1+s^2}}
\\&=
\frac{1}{\sqrt{|\l|}}\ln\left(\tfrac{\sqrt{|\l|}u_r}{v_-}+
\sqrt{\Big(\tfrac{\sqrt{|\l|}u_r}{v_-}\Big)^2+1}\,\right),
\end{split}
\end{equation}
because $g(s):=\ln(s+\sqrt{s^2+1})$ satisfies $g'(s)=\tfrac{1}{\sqrt{s^2+1}}$ for all $s\in\R$. Thus, according to \eqref{c.7}, the set of points $(u_r,v_r)$ for which
$$
   {T}_{\mathscr{L}}(\l,u_r,0)=h
$$
can be described through $g(\tfrac{\sqrt{|\l|}u_r}{v_-})=\sqrt{|\l|}h$, or, equivalently,
$$
 \frac{\sqrt{|\l|}u_r}{v_-}=g^{-1}(\sqrt{|\l|}h),
$$
because $g'(s)>0$ for all $s\in\R$, and, hence, it is globally invertible. Consequently, since  $\mathscr{E}_{\mathscr{L}}(u_r,v_r)=v_-^2$, it becomes apparent that
$$
\frac{v_r^2}{|\l|u_r^2}-1=\frac{v_-^2}{|\l|u_r^2}=[g^{-1}(\sqrt{|\l|}h)]^{-2}\equiv D(h).
$$
Therefore,
$$
v_r=\pm\sqrt{|\l| (D(h)+1)}\, u_r.
$$
In other words, setting $m_h:=\sqrt{|\l| (D(h)+1)}$, we have that
$$
  \mathscr{A}_h:= \{(u_r,v_r)\;:\;v_r=\pm m_h u_r\}
$$
In particular, the set of points $(u_r,v_r)$, with $u_r<0$, reaching the negative $v$-axis at a time $h$ are those on the straight line
\begin{equation}
\label{c.8}
v=-m_h u.
\end{equation}
An important feature  is that $m_h$ is decreasing with respect to $h$.
\par
On the other hand, by the proof of Theorem \ref{th2.2}(c), we already know that, for every $h\in (0,1)$,
$$
   x_{m_h} \in (0,\tfrac{1-h}{2}]\cup [\tfrac{1+h}{2},1)\;\; \hbox{if}\;\; u(x_{m_h})=\|u_h\|_\infty.
$$
Thus, there exists an increasing  subsequence, $\{h_n\}_{n\geq 1}$, in $(0,1)$, such that either
\begin{equation}
\label{c.9}
   \lim_{n\to \infty}h_n =1\;\; \hbox{and}\;\; x_n\equiv x_{m_{h_n}}\in  (0,\tfrac{1-h_n}{2}]\;\;
   \hbox{for all}\;\; n\geq 1,
\end{equation}
or
\begin{equation}
\label{c.10}
  \lim_{n\to \infty}h_n =1\;\; \hbox{and}\;\; x_n\equiv x_{m_{h_n}}\in  [\tfrac{1+h_n}{2},1)\;\;
   \hbox{for all}\;\; n\geq 1,
\end{equation}
though these alternatives are far from excluding, because the symmetric solutions satisfy both.
Assume \eqref{c.9} holds.  Then, owing to \eqref{c.6}, we find that
\begin{equation*}
\frac{v_{h_n}\left(\tfrac{1-h_n}{2}\right)}{u_{h_n}\left(\tfrac{1-h_n}{2}\right)}\geq -m_{h_1}.
\end{equation*}
Note that $v_{h_n}(\tfrac{1-h_n}{2})\leq0$. Indeed, by \eqref{c.9}, $v(x_n)=0$, and, hence, by the clockwise rotation sense in the phase plane, $v(x)\leq0$ for all $x\in[x_n,\tfrac{1-h_n}{2}]$. As $m_{h_1}>0$, the above inequality implies
\begin{equation}
\label{c.11}
v_{h_n}^2\left(\tfrac{1-h_n}{2}\right)\leq m_{h_1}^2u_{h_n}^2\left(\tfrac{1-h_n}{2}\right).
\end{equation}
Thus, as the associated energy to the system ($\mathscr{N}$) is
$$
\mathscr{E}_{\mathscr{N}}(u,v)=\frac{1}{2}v^2+\frac{\l}{2}u^2+\frac{1}{p+1}u^{p+1},
$$
it follows from \eqref{c.11} that
\begin{equation*}
\begin{split}
\mathscr{E}_{\mathscr{N}}(u_{h_n}\left(\tfrac{1-h_n}{2}\right),v_{h_n}& \left(\tfrac{1-h_n}{2}\right))=
\frac{\l}{2}\|u_{h_n}\|_{\infty}^2+\frac{1}{p+1}\|u_{h_n}\|_{\infty}^{p+1}\\&
=\frac{1}{2}v_{h_n}^2\left(\tfrac{1-h_n}{2}\right)+\frac{\l}{2}u_{h_n}^2
\left(\tfrac{1-h_n}{2}\right)+\frac{1}{p+1}u_{h_n}^{p+1}\left(\tfrac{1-h_n}{2}\right)
\\&\leq\frac{m_{h_1}^2+\l}{2}u_{h_n}^2\left(\tfrac{1-h_n}{2}\right)+\frac{1}{p+1}u_{h_n}^{p+1}
\left(\tfrac{1-h_n}{2}\right)\\&\leq \frac{m_{h_1}^2}{2}u_{h_n}^2\left(\tfrac{1-h_n}{2}\right)+\frac{1}{p+1}u_{h_n}^{p+1}
\left(\tfrac{1-h_n}{2}\right).
\end{split}
\end{equation*}
Therefore, we can conclude from \eqref{3.1} that
\begin{equation}
\label{c.12}
\lim_{n\to+\infty}u_{h_n}\left(\tfrac{1-h_n}{2}\right)=+\infty.
\end{equation}
On the other hand, by construction, we have that, for every $n\geq 1$,
\begin{equation}
\label{c.13}
-u_{h_n}''=\l u_{h_n}\;\; \hbox{in}\;\;\left(\tfrac{1-h_n}{2},\tfrac{1+h_n}{2}\right).
\end{equation}
Since $\l<0$, for every $n\geq 1$, there exist two constants, $A_{h_n}, B_{h_n}\in\R$, such that, for every $x\in \left(\tfrac{1-h_n}{2},\tfrac{1+h_n}{2}\right)$,
\begin{equation*}
u_{h_n}(x)=A_{h_n}e^{\sqrt{|\l|}x}+B_{h_n}e^{-\sqrt{|\l|}x}=
e^{-\sqrt{|\l|}x}A_{h_n}\left(e^{2\sqrt{|\l|}x}+\tfrac{B_{h_n}}{A_{h_n}}\right).
\end{equation*}
Equivalently, for every $x\in (0,1)$ and sufficiently large $n\geq 1$,
\begin{equation}
\label{c.14}
e^{\sqrt{|\l|}x} u_{h_n}(x)=A_{h_n}\left(e^{2\sqrt{|\l|}x}+\tfrac{B_{h_n}}{A_{h_n}}\right).
\end{equation}
Note that, for every $x\in (0,1)$ and sufficiently large $n$, we have that $x\in\left(\tfrac{1-h_n}{2},\tfrac{1+h_n}{2}\right)$.
Thanks to \eqref{c.12}, we can infer from \eqref{c.14} that
\begin{equation}
\label{c.15}
\begin{split}
  +\infty & = \lim_{n\to+\infty} \left( e^{\sqrt{|\l|}\frac{1-h_n}{2}}u_{h_n}(\tfrac{1-h_n}{2})\right) \\
  &=\lim_{n\to+\infty}\left(A_{h_n}\left(e^{2\sqrt{|\l|}\frac{1-h_n}{2}}+\tfrac{B_{h_n}}{A_{h_n}}\right)\right).
  \end{split}
\end{equation}
Thus, along some subsequence, also labeled  by $n$, we find that, either
\begin{equation}
\label{c.16}
  \lim_{n\to+\infty}A_{h_n}=+\infty, \;\;\hbox{or}\;\;  \lim_{n\to+\infty}B_{h_n}=+\infty.
\end{equation}
Suppose that
$$
   \liminf_{n\to+\infty}A_{h_n}>-\infty.
$$
Then, by \eqref{c.14} and \eqref{c.15}, we find that, for every $x\in (0,1)$,
\begin{align*}
\liminf_{n\to+\infty}\left( e^{\sqrt{|\l|}x} u_{h_n}(x)\right) = & \liminf_{n\to+\infty}\left( e^{\sqrt{|\l|}\frac{1-h_n}{2}} u_{h_n}(\tfrac{1-h_n}{2})\right)\\ &
+\liminf_{n\to+\infty}\left(A_{h_n}\left(e^{2\sqrt{|\l|}x}-e^{2\sqrt{|\l|}\frac{1-h_n}{2}}\right)\right)=+\infty.
\end{align*}
Consequently,
$$
  \lim_{n\to+\infty}u_{h_n}(x)=+\infty.
$$
Now, suppose that $\liminf_{n\to+\infty}A_{h_n}=-\infty$. Then, along some subsequence, labeled again by $n$,
we can assume that
\begin{equation}
\label{c.17}
   \lim_{n\to+\infty}A_{h_n}=-\infty
\end{equation}
and, thanks to \eqref{c.16},
$$
   \lim_{n\to+\infty}B_{h_n}=+\infty.
$$
In this case, by \eqref{c.14} and the positivity of $u_{h_n}$, we find that, for every $x\in (0,1)$ and sufficiently large $n$,
\begin{equation*}
  e^{2\sqrt{|\l|}x}+\tfrac{B_{h_n}}{A_{h_n}}<0.
\end{equation*}
We claim that, for every $x\in (0,1)$,
\begin{equation}
\label{c.18}
  \limsup_{n\to +\infty} \left( e^{2\sqrt{|\l|}x}+\tfrac{B_{h_n}}{A_{h_n}}\right) <0.
\end{equation}
To show \eqref{c.18} we will argue by contradiction, assuming that there exist  $\tilde x\in (0,1)$ and a subsequence of $h_n$, relabeled by $n$, such that
$$
\lim_{n\to+\infty}\left(e^{2\sqrt{|\l|}\tilde x}+\tfrac{B_{h_n}}{A_{h_n}}\right)=0.
$$
Then, by \eqref{c.17}, for every $x>\tilde x$, we have that
$$
\lim_{n\to+\infty}\left( e^{\sqrt{|\l|} x}u_{h_n}(x)\right) =\lim_{n\to+\infty}\left( A_{h_n}\left(e^{2\sqrt{|\l|}x}+\tfrac{B_{h_n}}{A_{h_n}}\right)\right)=-\infty,
$$
contradicting the positivity of $u_{h_n}$. Therefore, again by \eqref{c.17}, it becomes apparent
from \eqref{c.18} that, for every $x\in (0,1)$,
$$
\lim_{n\to+\infty}\left( e^{\sqrt{|\l|} x} u_{h_n}(x)\right) =
\lim_{n\to+\infty}\left( A_{h_n}\left(e^{2\sqrt{|\l|}x}+\tfrac{B_{h_n}}{A_{h_n}}\right)\right)=+\infty.
$$
The case when \eqref{c.10} occurs can be proven easily by adapting
the proof of the case \eqref{c.9}. By repetitive, its technical details will be omitted here.
As these arguments can be repeated along any increasing sequence of $h_n$'s converging to $1$ as $n\to +\infty$, the proof is complete.
\end{proof}

The next result establishes that the solutions of the component $\mathscr{C}_{1}^+$ perturb into solutions of  the components $\mathscr{C}_{h,s}^+$ as $h<1$ perturbs from $1$.

\begin{theorem}
\label{th3.2}
For every $\a>0$, there are two increasing sequences, $\{h_n\}_{n\geq1}$ in $(0,1)$ and $\{\l_n\}_{n\geq1}$ in $(0,\pi^2)$, such that
\begin{equation}
\label{c.19}
\lim_{n\to+\infty}h_n=1, \quad \lim_{n\to+\infty}\l_n=\pi^2,\quad \hbox{and}\;\; \|u_{\l_n}\|_\infty =\a \;\;
\hbox{for all}\;\; n\geq 1,
\end{equation}
where $u_{\l_n}$ is the unique symmetric positive solution of \eqref{1.1} for
$(\l,h)=(\l_n,h_n)$, $n\geq 1$. Moreover,
\begin{equation*}
\lim_{n\to+\infty} u_{\l_n}=\alpha\sin(\pi \cdot)\quad  \hbox{in}\;\; \mc{C}^1([0,1];\R).
\end{equation*}
In other words,  $(\pi^2,\a \sin(\pi \cdot))\in \mathscr{C}_1^+$ perturbs into the sequence
$$
  (\l_n,u_{\l_n})\in \mathscr{C}_{h_n,s}^+,\quad n\geq 1.
$$
\end{theorem}
\begin{proof}
According to Theorem \ref{th2.2}, for every $h\in (0,1)$, $(\l,u_\l)$ bifurcates from $u=0$ at $\l=\pi^2$,
$\mc{P}_\l(\mathscr{C}_{h,s}^+)=(-\infty,\pi^2]$ and \eqref{ii.28} is satisfied.
\par
Thanks to Theorem \ref{th3.1}, for every integer $n\geq 1$, we have that
$$
  \lim_{h\uparrow 1}\|u_{\pi^2-\frac{1}{n}}\|_\infty =\infty.
$$
Thus, there exists $h_n \in (1-\frac{1}{n},1)$ such that
$$
   \|u_{\pi^2-\frac{1}{n}}\|_\infty>\a \quad \hbox{for}\;\; h=h_n.
$$
Hence, thanks to Theorem \ref{th2.2}, by the continuity of $\mathscr{C}_{h_n,s}^+$, there exists
$\l_n\in (\pi^2-\frac{1}{n},\pi^2)$ such that
$$
    \|u_{\l_n}\|_\infty=\a\quad \hbox{for}\;\; h=h_n.
$$
By construction, the sequence $(h_n,\l_n,u_{\l_n})$, $n\geq 1$, satisfies \eqref{c.19}. Moreover,
for every $n\geq 1$,
\begin{equation}
\label{c.20}
\frac{u_{n}}{\|u_{n}\|_{\infty}}
=\mc{K}\left(\l_n\frac{u_{n}}{\|u_{n}\|_{\infty}}+a_{h_n}u_{n}^{p-1}\frac{u_{n}}{\|u_{n}\|_{\infty}}\right),
\end{equation}
where $\mc{K}$ is the resolvent operator introduced in the proof of Theorem \ref{th3.1}. Since $\|u_{n}\|_{\infty}=\a$ for all $n\geq 1$, and $\lim_{n\to\infty}\l_n=\pi^2$, the sequence
\[
\l_n\frac{u_{n}}{\|u_{n}\|_{\infty}}+a_{h_n}u_{n}^{p-1}\frac{u_{n}}{\|u_{n}\|_{\infty}},\quad n\geq 1,
\]
is bounded in $[0,1]$. Thus, by the compactness of $\mc{K}$, it follows from \eqref{c.20}
that, along some subsequence, relabeled by $n_m$, $m\geq 1$,
$$
    \lim_{m\to+\infty}\frac{u_{{n_m}}}{\|u_{{n_m}}\|_\infty}= \psi\quad \hbox{in}\;\; \mc{C}_0^1([0,1];\R)
$$
for some $\psi\in \mc{C}_0^1([0,1];\R)$ such that $\|\psi\|_{\infty}=1$ and $\psi\gneq 0$. Finally, letting
$m\to \infty$ in \eqref{c.20} with $n=n_m$, it becomes apparent that
\[
\left\{ \begin{array}{l}
	-\psi''=\pi^2\psi\quad\hbox{in}\;(0,1),\\[1pt]
	\psi(0)=\psi(1)=0, \end{array} \right.
\]
which implies $\psi(x)=\sin(\pi x)$ for all $x\in[0,1]$. Therefore,
$$
    \lim_{m\to+\infty}\frac{u_{{n_m}}}{\a}= \sin(\pi\cdot)\quad \hbox{in}\;\; \mc{C}_0^1([0,1];\R).
$$
As the same argument can be repeated along any subsequence, we find that
\[
\lim_{n\to+\infty}\frac{u_n}{\a}= \sin(\pi\cdot)\quad \hbox{in}\;\; \mc{C}_0^1([0,1];\R).
\]
This ends the proof.
\end{proof}

\setcounter{equation}{0}
\section{A multiplicity result for $\l<0$}
\label{sec4}

\noindent As already discussed in Section \ref{sec2}, when $\l<0$ in \eqref{2.1}, the origin of ($\mathscr{N}$) is a local saddle point whose local unstable and stable manifolds conform the homoclinic connection of the level set $\mathscr{E}_{\mathscr N}(u,v)=0$, as illustrated in the left plot of Figure \ref{Fig1}. Moreover, the origin of $(\mathscr L)$ in \eqref{2.1} is a global saddle point whose unstable and stable
manifolds are the straight lines
\begin{equation}
\label{d.1}
    W^u:\; v=\sqrt{-\l}\,u\quad\hbox{and}\quad W^s:\; v=-\sqrt{-\l}\,u,
\end{equation}
respectively. Recalling that
$$
\mathscr{E}_{\mathscr{N}}(u,v):=\frac{1}{2}\left(v^2+\l u^2\right)+\frac{1}{p+1}u^{p+1},
$$
the tangent cone of $\mathscr{E}_{\mathscr{N}}(u,v)$ at the origin is given by the zeros of the homogeneous part $v^2+\l u^2$. Denoting it by $V(v^2+\l u^2)$, it is apparent that
\begin{equation}
	\label{d.2}
	V(v^2+\l u^2)=\{(u,v)\,:\,\mathscr{E}_{\mathscr{L}}(u,v)=0\}=W^u\cup W^s\cup\{(0,0)\}.
\end{equation}
Subsequently, we consider the time map  $T_{\mathscr N}(\l,u_0)$ introduced in \eqref{2.4} and defined for all $u_0>u_{\rm ho}$. Since $T_{\mathscr N}(\l,u_0)$ is decreasing with respect to $u_0>u_{\rm ho}$, it follows from \eqref{2.5} that there exists $v_0>0$ such that
\begin{equation}
\label{d.3}
	T_{\mathscr N}(\l,u_0)=\frac{1-h}{4} \;\; \hbox{with}\;\; v_0^2=\l u_0^2+\frac{2}{p+1}u_0^{p+1},
\end{equation}
i.e. the orbit of ($\mathscr{N}$) through $(0,v_0)$ needs a time $2T_{\mathscr N}(\l,u_0)=\frac{1-h}{2}$ to reach $(0,-v_0)$. Actually, by the monotonicity of $T_{\mathscr N}(\l,u)$ with respect to $u$, for every
$v\in (0,v_0]$, the unique solution of ($\mathscr{N}$) through $(0,v)$  remains in either the first or the fourth quadrants in the time interval $[0,\frac{1-h}{2}]$. Subsequently, we denote by $\Phi_{\mathscr{N}}^+(x,\cdot)$ the flow map of ($\mathscr{N}$) for all $x>0$, and, setting $\Gamma_0:=\{0\}\times [0,v_0]$, we consider the transformed curve by the flow of the vertical segment $\G_0$ after time $\frac{1-h}{2}$, denoted by
\begin{equation*}
	\Gamma_{0,\frac{1-h}{2}}^+:=\Phi_{\mathscr{N}}^+\left(\tfrac{1-h}{2},\Gamma_0\right),
\end{equation*}
which has been plotted using blue light color in the right picture of Figure \ref{Fig5}.
By construction, the curve $\Gamma_{0,\frac{1-h}{2}}^+$ lies in the first and fourth quadrants, and
$$
   \Phi_{\mathscr{N}}^+\left(\tfrac{1-h}{2},(0,v_0)\right)=(0,-v_0),
$$
as illustrated in the right plot of Figure \ref{Fig5}. Thus,  $\Gamma_{0,\frac{1-h}{2}}^+$  links the points $(0,0)$ and $(0,-v_0)$. Moreover,  by the monotonicity of ${T}_{\mathscr N}$,  there exists a unique $v_1\in (0,v_0)$ such that
\begin{equation}
	\label{d.4}
	\Phi_{\mathscr N}^+\left(\tfrac{1-h}{2},(0,v_1)\right)=(u_1,0)\in\Gamma_{0,\frac{1-h}{2}}^+.
\end{equation}
According to \eqref{2.4}, \eqref{d.3} and \eqref{d.4}, $u_0$ and $u_1$ satisfy
\begin{equation}
	\label{d.5}
	\begin{split}
		 \frac{1-h}{2}& =2 T_{\mathscr{N}}(\l,u_0)=2\int_0^1\frac{ds}{\sqrt{\l(1-s^2)+\frac{2}{p+1}u_0^{p-1}(1-s^{p+1})}},\\
		 \frac{1-h}{2} & =T_{\mathscr{N}}(\l,u_1)=\int_0^1\frac{ds}{\sqrt{\l(1-s^2)+\frac{2}{p+1}u_1^{p-1}(1-s^{p+1})}}.
	\end{split}
\end{equation}
The time map  $T_{\mathscr N}$ can be further decomposed as follows. For any given  $(u_\o,v_\o)\in W^u$, with $u_\o>0$, by \eqref{d.1}, we have that  $v_\o=\sqrt{-\l}u_\o>0$. Next, we consider the orbit of ($\mathscr{N}$)
$$
   \mathscr{E}_{\mathscr{N}}(u,v)=\mathscr{E}_{\mathscr{N}}(u_\o,v_\o),
$$
departing from $\G_0$ at, say, $(0,v_+)$ and crossing the $u$-axis at, say,  $(u_+,0)$.  Then, since
$$
  v^2+\l u^2 +\frac{2}{p+1}u^{p+1}=v^2_\o+\l u^2_\o +\frac{2}{p+1}u^{p+1}_\o=\frac{2}{p+1}u^{p+1}_\o,
$$
the necessary time to connect $(0,v_+)$ with $(u_\o,v_\o)$ along that orbit, is given by
\begin{equation}
\label{d.6}
  T_{\mathscr{N}}(\l,0,u_\o)=\int_0^1\frac{ds}{\sqrt{-\l s^2+\frac{2}{p+1}u_\o^{p-1}(1-s^{p+1})}}.
\end{equation}
Similarly, since
\begin{equation}
\label{d.7}
  v^2+\l u^2 +\frac{2}{p+1}u^{p+1}=\frac{2}{p+1}u^{p+1}_\o=\l u_+^2+\frac{2}{p+1}u^{p+1}_+,
\end{equation}
the necessary time to connect $(u_\o,v_\o)$ with $(u_+,0)$ is
\begin{equation}
\label{d.8}
T_{\mathscr{N}}(\l,u_\o,u_+)=\int_{\frac{u_\o}{u_+}}^1\frac{ds}{\sqrt{\l(1-s^2)+\frac{2}{p+1}u_+^{p-1}(1-s^{p+1})}}.
\end{equation}
By construction,
\begin{equation}
\label{d.9}
   T_{\mathscr{N}}(\l,0,u_+)=T_{\mathscr{N}}(\l,0,u_\o)+T_{\mathscr{N}}(\l,u_\o,u_+).
\end{equation}
Since $T_{\mathscr{N}}(\l,0,u_\o)$ is decreasing with respect to $u_\o$, the intersection
$\Gamma_{0,\frac{1-h}{2}}^+\cap W^u$ consists of a single point. Moreover, thanks to \eqref{d.7}, we have that
\[
\frac{u_\o}{u_+}=\left((p+1)\tfrac{\l}{2}u_+^{1-p}+1\right)^{\frac{1}{p+1}}.
\]
Thus, since  $p>1$ and $\l<0$, it is apparent that $\frac{u_\o}{u_+}$ is increasing with respect to $u_+$. Consequently, $T_{\mathscr{N}}(\l,u_\o,u_+)$ is decreasing with respect to $u_+$. Therefore, by the symmetry of $(\mathscr{N})$ and $(\mathscr{L})$ with respect to the $u$-axis, the curve $\Gamma_{0,\frac{1-h}{2}}^+$ also  intersects  $W^s$ at a single point,  as illustrated by the right plot of Figure \ref{Fig5}.
\par
Similarly, setting $\Gamma_1:=\{0\}\times [-v_0,0]$, we can shot backwards in time by the reversed flow $\Phi_{\mathscr{N}}^-(x,\cdot)$ of $(\mathscr{N})$ moving from $x=1$ to $x=\frac{1+h}{2}$. This provides us with the curve
\begin{equation*}
	\Gamma_{1,\frac{1+h}{2}}^-:=\Phi_{\mathscr{N}}^-(\tfrac{1+h}{2},\Gamma_1),
\end{equation*}
which has been plotted using green color on the right picture of Figure \ref{Fig5}. By construction, the curve $\Gamma_{1,\frac{1+h}{2}}^-$ lies entirely in the first and fourth quadrants. Moreover, by the symmetries of $a(x)$ and the systems $(\mathscr{N})$ and $(\mathscr L)$, it is apparent that
$$
\Gamma_{1,\frac{1+h}{2}}^-=\left\{(u,v)\in\mathbb{R}^2\,:\,(-u,v)\in \Gamma_{0,\frac{1-h}{2}}^+\right\}.
$$

\begin{figure}[h!]
	\centering
	\includegraphics[scale=0.9]{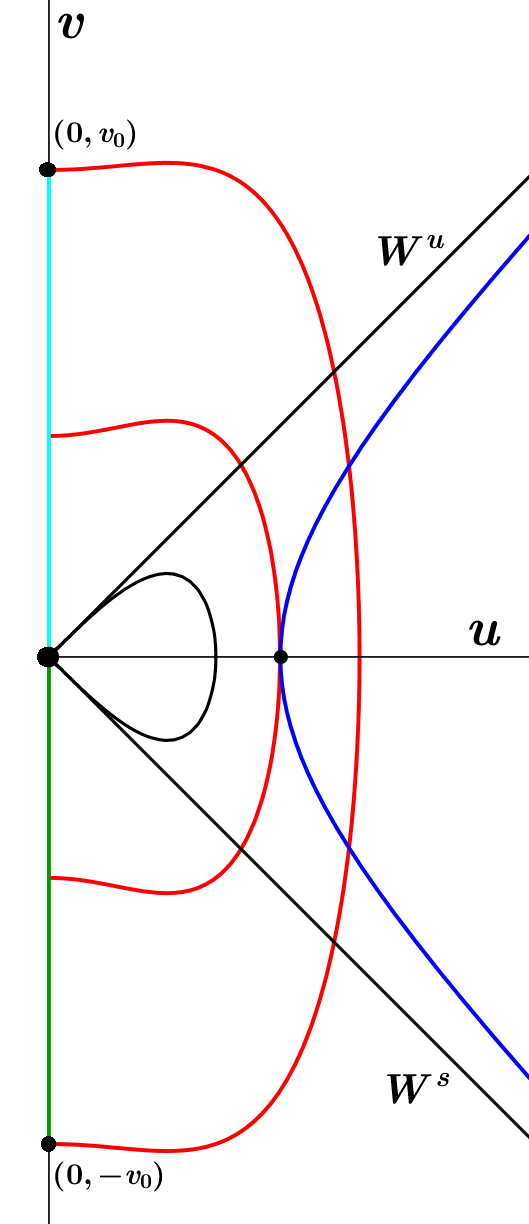}\qquad\qquad\qquad\includegraphics[scale=0.9]{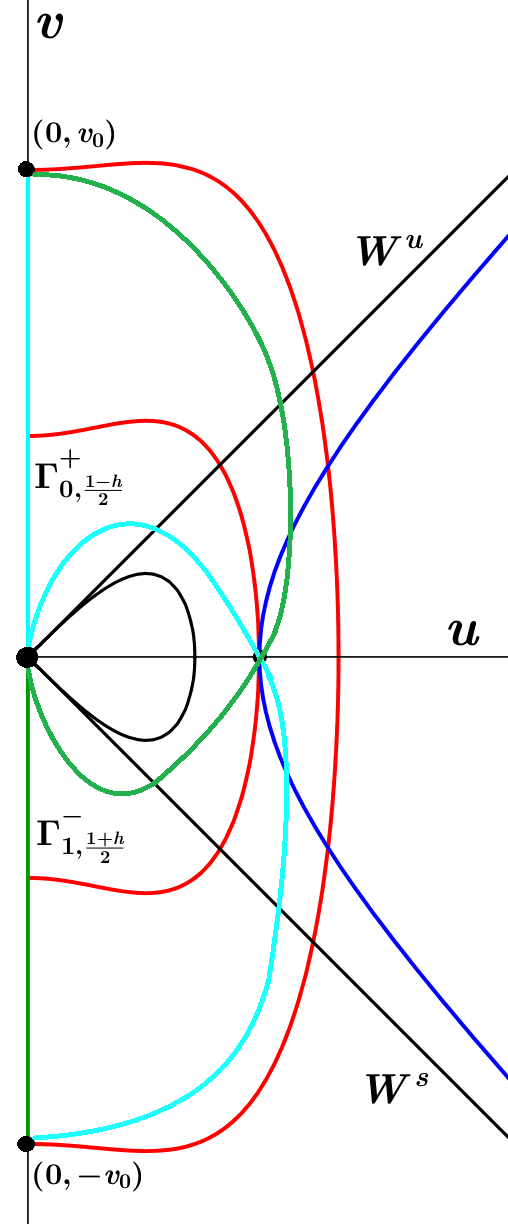}
	\caption{The sets $\Gamma_0$ in blue light and $\Gamma_1$ in green (left), as well as the sets $\Gamma_{0,\frac{1-h}{2}}^+$ in blue light and $\Gamma_{1,\frac{1+h}{2}}^-$ in green (right).}
	\label{Fig5}
\end{figure}

In particular, it crosses $W^{u}$ and $W^s$ at two single points, much like  $\Gamma_{0,\frac{1-h}{2}}^+$.
Figure \ref{Fig5} sketches the main properties of these curves. Note that
$$
  \Gamma_{1,\frac{1+h}{2}}^-\cap \Gamma_{0,\frac{1-h}{2}}^+=\{(u_1,0)\}
$$
Obviously, to construct positive asymmetric solutions of the problem \eqref{1.1}, the curves $\Gamma_{0,\frac{1-h}{2}}^+$ and $\Gamma_{1,\frac{1+h}{2}}^-$ should be connected in a time $h\in (0,1)$
through some orbit of the linear system $(\mathscr{L})$. The next result shows how to do it.

\begin{theorem}
\label{th4.1}
For every $h\in(0,1)$, there exists $\tilde \l(h)<0$ such that \eqref{1.1} has, at least,
two asymmetric positive solutions for each $\l\leq\tilde\l(h)$.
\end{theorem}

\begin{proof}	
According the previous analysis, the set
\begin{equation}
\label{d.10}
		\Gamma_{0,\frac{1-h}{2}}^+\cap \{(u,v)\;:\;\;v\geq\sqrt{-\l}u\geq 0\}
\end{equation}
is the piece  of $\Gamma_{0,\frac{1-h}{2}}^+$ in the first quadrant connecting $(0,0)$ with $W^{u}$.
Similarly,
\begin{equation}
\label{d.11}
		\Gamma_{1,\frac{1+h}{2}}^-\cap \{(u,v)\;:\;\; v \geq\sqrt{-\l}u\geq 0 \}
\end{equation}
is the piece of $\Gamma_{1,\frac{1+h}{2}}^-$ connecting $(0,v_0)$ with $W^{u}$. According to the nature of
the flow of $(\mathscr{L})$, there exists $v_*>0$ such that, for every  $\tilde v\in(0,v_*)$, the level line of the linear system
\begin{equation}
\label{d.12}
		\mc{E}_{\mathscr{L}}(u,v)=\frac{1}{2}v^2+\frac{\lambda}{2}u^2=\frac{1}{2}\tilde v^2
\end{equation}
intersects both \eqref{d.10} and \eqref{d.11}. By continuous dependence, the necessary time to connect \eqref{d.10} with \eqref{d.11} through \eqref{d.12} grows to infinity as $\tilde v\downarrow 0$.
	\par
On the other hand, in the portion of the first quadrant in between the $u$-axis and  $W^u$, the sets  \begin{equation}
\label{d.13}
		\Gamma_{0,\frac{1-h}{2}}^+\cap \{(u,v)\;:\;\; 0\leq v\leq\sqrt{-\l}u\}
\end{equation}
and
\begin{equation}
\label{d.14}
		\Gamma_{1,\frac{1+h}{2}}^-\cap \{(u,v)\;:\;\; 0\leq v\leq\sqrt{-\l}u\}	
\end{equation}
are the two arcs of curve of $\Gamma_{0,\frac{1-h}{2}}^+$ and $\Gamma_{1,\frac{1+h}{2}}^-$, respectively, connecting the portion of $W^{u}$ in the first quadrant with the point $(u_1,0)$ defined in \eqref{d.5}. Note that
$$
   (u_1,0)\in \Gamma_{1,\frac{1-h}{2}}^+\cap \Gamma_{1,\frac{1+h}{2}}^-.
$$
By construction, the interval of times necessary to connect \eqref{d.13} to \eqref{d.14} with the orbits of ($\mathscr{L}$) passing through some point $(\tilde u,0)$ with $\tilde u \in (0,u_1)$ is bounded. Therefore, by the continuity of the time map associated to $(\mathscr{L})$, the connection times between $\Gamma_{0,\frac{1-h}{2}}^+$ and $\Gamma_{1,\frac{1+h}{2}}^-$ within the first quadrant cover an interval of the form $[\a,+\infty)$ for some $\a>0$; $\a$ being the minimum connection time between these curves. To complete the proof of the theorem it suffices to show that $\a<h$ for sufficiently large $|\l|=-\l$.
\par
Let $(\hat u,\hat v)$ denote the crossing point of \eqref{d.14} with the orbit of ($\mathscr{L}$)
through the point $(u_1,0)$
\begin{equation}
\label{d.15}
	\mathscr{E}_{\mathscr{L}}(u,v)=\frac{1}{2}v^2+\frac{\lambda}{2}u^2=\frac{\l}{2}u_1^2,
\end{equation}
which is the crossing point between the green and the blue lines on the right picture of Figure \ref{Fig5}.
The connection time between $(u_1,0)$ and $(\hat u,\hat v)$  along \eqref{d.15} is given by
$$
T_{\mathscr L}(\l,u_1,\hat u)=\int_{u_1}^{\hat u}\frac{du}{\sqrt{\l(u_1^2-u^2)}}=\frac{1}{\sqrt{-\l}}\int_1^{\frac{\hat u}{u_1}}\frac{d\t}{\sqrt{-1+\t^2}},
$$
or, equivalently,
\begin{equation}
\label{d.16}
T_{\mathscr L}(\l,u_1,\hat u) =
\tfrac{1}{\sqrt{-\l}}\ln\Big(\tfrac{\hat u}{u_1}+\sqrt{\left(\tfrac{\hat{u}}{u_1}\right)^2-1}\Big).
\end{equation}
Next, we will show that the quotient $\frac{\hat u}{u_1}$ is bounded above. Indeed, coming back to the definition  of $u_0$ and $u_1$ in \eqref{d.5}, it is apparent that $u_0>u_1$. Moreover, since the denominator of the integrands in \eqref{d.5},
$$
   \sqrt{\l(1-s^2)+\tfrac{2}{p+1}u_1^{p-1}(1-s^{p+1})}
$$
is defined for all $s\in (0,1)$, we find that  $\l+\tfrac{2}{p+1}u_1^{p-1}\geq 0$ and, hence,
\begin{equation}
\label{d.17}
	u_0>u_1\geq \left(\tfrac{p+1}{2}\right)^{\frac{1}{p-1}}(-\l)^{\frac{1}{p-1}}.
\end{equation}
Subsequently, for every $\l<0$, we set
$$
   \kappa_0(\l):= u_0 (-\l)^{\frac{-1}{p-1}},\quad  \kappa_1(\l):= u_1 (-\l)^{\frac{-1}{p-1}}.
$$
Note that, by construction, $u_0$ and $u_1$ also vary with $\l$, as the parameter $\l$ appears in the
formulation of \eqref{1.1}. Then,
$$
        u_0=\kappa_0(\l)(-\l)^{\frac{1}{p-1}},\quad u_1=\kappa_1(\l)(-\l)^{\frac{1}{p-1}},
$$
and, thanks to \eqref{d.17},
$$
   \kappa_0(\l)>\kappa_1(\l)>\left(\tfrac{p+1}{2}\right)^{\frac{1}{p-1}}.
$$
On the other hand, it follows from \eqref{d.5} that
\begin{equation*}
\begin{split}
 \frac{1-h}{2}\,\sqrt{-\l}&= \int_0^1\frac{ds}{\sqrt{s^2-1+\frac{2}{p+1}\kappa_1^{p-1}(\l)(1-s^{p+1})}}
\\&=2 \int_0^1\frac{ds}{\sqrt{s^2-1+\frac{2}{p+1}\kappa_0^{p-1}(\l)(1-s^{p+1})}}.
\end{split}
\end{equation*}
As these integrals are decreasing with respect to $\kappa_0(\l)$ and $\kappa_1(\l)$, $\frac{1-h}{2}\sqrt{-\l}$ decreases with $\l<0$ for fixed $h$, and
$$
\lim_{\t\uparrow +\infty}  \int_0^1\frac{ds}{\sqrt{s^2-1+\theta (1-s^{p+1})}}=0,
$$
it becomes apparent that, for every $\l_*<0$, there exists a constant $M(\l_*)>0$ such that
$$
	\left(\tfrac{p+1}{2}\right)^{\frac{1}{p-1}}<\kappa_1(\l)<\k_0(\l)<M(\l_*)\;\;
\hbox{for all}\;\; \l\in (-\infty,\l_*].
$$
Thus,
\begin{equation}
\label{d.18}
		\frac{\hat u}{u_1}<\frac{u_0}{u_1}=\frac{\kappa_0(\l)}{\kappa_1(\l)}
      <M(\l_*)\left(\frac{2}{p+1}\right)^{\frac{1}{p-1}}\qquad\hbox{for all}\;\l\leq\l_*.
\end{equation}
Therefore, by \eqref{d.18}, it follows from \eqref{d.16} that there exists $\tilde\l(h)$ such that
$$
	T_{\mathscr L}(\l,u_1,\hat u)<h\;\, \hbox{for all}\;\; \l\in (-\infty,\tilde\l(h)].
$$
The second asymmetric solution is obtained either by reflection around $x=0.5$ of the one we have just constructed, or repeating the previous argument in the fourth quadrant. This ends the proof.
\end{proof}

\setcounter{equation}{0}

\section{A multiplicity result for $\l>0$}
\label{sec5}

\noindent In this section, we analyze the existence of, at least, two positive asymmetric solutions of \eqref{1.1} when $\l>0$ and $h$ is sufficiently close to $1$. As in Section \ref{sec2.2}, to treat the case $\l>0$ it is better to consider, instead of \eqref{2.2}, the equivalent systems
\begin{equation}
\label{5.1}
	(\mathscr{N}_\mathrm{eq})\;\; \left\{
	\begin{array}{ll}
		u'=\sqrt{\l}\,v,\\[3pt]
		v'=-\sqrt{\l}\,u-\frac{1}{\sqrt{\l}}\, u^{p},
	\end{array}
	\right.\qquad
	(\mathscr{L}_\mathrm{eq})\;\; \left\{
	\begin{array}{ll}
		u'=\sqrt{\l}\,v,\\[3pt]
		v'=-\sqrt{\l}\,u.
	\end{array}
	\right.
\end{equation}
If $\l>0$, there is no any restriction on the location of the maxima of the positive
solutions of \eqref{1.1} in the interval $(0,1)$. According to Section \ref{sec2.2}, the unique symmetric solution  reaches its maximum at $x=0.5$, where $a(0.5)=0$ always.
\par
However, for asymmetric positive solutions, one can study separately the case where the maximum is reached at $a^{-1}(0)$ and the case where it is attained at $a^{-1}(1)$. In this paper, we restrict ourselves to deal with the special case where, for every positive solution $u$ of \eqref{1.1},
\begin{equation}
	\label{5.2}
	x_m\in a^{-1}(0)=(\tfrac{1-h}{2},\tfrac{1+h}{2}),\quad\hbox{where}\;\; u(x_m)=\|u\|_{\infty},
\end{equation}
because this will provide us with two asymmetric solutions for $\l<\pi^2$ arbitrarily close to
$\pi^2$ for  $h$  sufficiently close to $1$, and this is the first multiplicity result of this nature
available in the literature; suggested  by the numerical experiments of \cite{CLGT-2024}. In such a case, we should first connect a point in the positive $v$-semiaxis with a point in the first quadrant through an energy level line of $(\mathscr{N}_{\rm eq})$ in the time interval $[0,\tfrac{1-h}{2}]$. For every $\t_0\in(0,\pi/2)$, the necessary time to link a point $(0,v_0)$ with a point $(u_{\t_0},v_{\t_0})$ lying in the line $v= (\tan\t_0)u$ through the orbit of $(\mathscr{E}_{\mathscr{N}_{\rm eq}})$
$$
v^2+u^2+\frac{2}{\l(p+1)}u^{p+1}=u_{\t_0}^2 \tan^2\t_0 +u_{\t_0}^2+\frac{2}{\l(p+1)}u_{\t_0}^{p+1}
$$
is given  by
\begin{equation*}
T_{\mathscr{N}}(\l,0,u_{\t_0}):=\frac{1}{\sqrt{\l}}\int_0^1\frac{ds}
{\sqrt{\tan^2\t_0+1-s^2+\frac{2u_{\t_0}^{p-1}}{\l(p+1)}(1-s^{p+1})}}
\end{equation*}
(see \eqref{ii.15}). Since this connection should be accomplished in time $\tfrac{1-h}{2}$, we impose that
\begin{equation}
	\label{5.3}
	T_{\mathscr{N}}(\l,0,u_{\t_0})=\frac{1-h}{2}.
\end{equation}
Secondly, by \eqref{5.2} and for some $\t_1\in(0,\pi/2)$, the point $(u_{\t_0},v_{\t_0})$ reached at time $x=\tfrac{1-h}{2}$ should be connected at time $x=\tfrac{1+h}{2}$ with a point of the fourth quadrant $(u_{\t_1},v_{\t_1})$ lying in the line $v=-(\tan\t_1)u$, through the energy line of linear system $(\mathscr{L}_{\rm eq})$
\begin{equation*}
	v^2+u^2=v_{\t_0}^2+u_{\t_0}^2=v_{\t_1}^2+u_{\t_1}^2=\|u\|_{\infty}^2,
\end{equation*}
which implies
\begin{equation}
	\label{5.4}
	\cos\t=\frac{u_\t}{\|u\|_{\infty}}\qquad\hbox{if}\;\;\t\in\{\t_1,\t_2\}.
\end{equation}
Since the angular variable component in ($\mathscr{L}_{\rm eq}$) satisfies $\t'(x)=-\sqrt{\l}$ (see Section \ref{sec2.2}), it follows that
\begin{equation*}
	-\t_1-\t_0=\int_{\frac{1-h}{2}}^{\frac{1+h}{2}}\t'(x)\, dx=-\int_{\frac{1-h}{2}}^{\frac{1+h}{2}}\sqrt{\l}\, dx=-h\sqrt{\l},
\end{equation*}
which is equivalent to ask for
\begin{equation}
	\label{5.5}
	\frac{\t_0+\t_1}{\sqrt{\l}}=h.
\end{equation}
Finally, by the symmetry of $a_h(x)$, in order to connect $(u_{\t_1},v_{\t_1})$ with the negative $v$-semiaxis,
\begin{equation}
	\label{5.6}	 T_{\mathscr{N}}(\l,0,u_{\t_1})=\frac{1}{\sqrt{\l}}\int_0^1\frac{ds}{\sqrt{\tan^2\t_1+1-s^2+\frac{2u_{\t_1}^{p-1}}{\l(p+1)}(1-s^{p+1})}}=\frac{1-h}{2}
\end{equation}
should hold. By \eqref{5.4} and recalling that
$$
   \tan^2\t+1=\tfrac{1}{\cos^2\t},
$$
integrals $T_{\mathscr{N}}(\l,0,u_{\t_0})$ and $T_{\mathscr{N}}(\l,0,u_{\t_1})$ can be rewritten for $\t\in\{\t_0,\t_1\}$ as
\begin{equation}
	\label{5.7}
	\begin{split}
		T_{\mathscr{N}}(\l,0,u_\t)&=\frac{1}{\sqrt{\lambda}}
		\int_{0}^{1}\frac{ds}{\sqrt{\frac{1}{\cos^2\t} -s^2 +
				\frac{2\|u\|_{\infty}^{p-1} \cos^{p-1}\t}{\lambda(p+1)}\left(1-s^{p+1}\right)}}\\&=\frac{\|u\|_{\infty}^{-\frac{p-1}{2}}}{\sqrt{\l}}\int_{0}^{1}\frac{ds}{\sqrt{\frac{1-s^2\cos^2\t}{\|u\|_{\infty}^{p-1}\cos^2\t}+
				\frac{2 \cos^{p-1}\t}{\lambda(p+1)}\left(1-s^{p+1}\right)}}.
	\end{split}
\end{equation}
Subsequently, we will denote
$$
   {S}_{\mathscr{N}}(\l,\|u\|_\infty,\t):= {T}_{\mathscr{N}}(\l,0,u_\t)\quad \hbox{if}\;\;
   \t\in\{\t_0,\t_1\}.
$$
According to \eqref{5.3}, \eqref{5.5}, \eqref{5.6} and \eqref{5.7}, for any fixed $\l\in(0,\pi^2]$, the existence of a positive solution of \eqref{1.1} satisfying \eqref{5.2} is guaranteed if, and only if,
\begin{equation}
	\label{5.8}	
	\begin{cases}
		\displaystyle\frac{\t_0+\t_1}{\sqrt{\l}}= h,\\[5pt]
		\displaystyle S_{\mathscr{N}}(\l,\|u\|_{\infty},\t_0)= S_{\mathscr{N}}(\l,\|u\|_{\infty},\t_1)=\frac{1-h}{2}.
	\end{cases}
\end{equation}
A solution of \eqref{5.8} with $\t_0=\t_1$ corresponds to a symmetric positive solution of \eqref{1.1}, whose existence and uniqueness has been already established by Theorem \ref{th2.1}.
On the other hand, when \eqref{5.8} holds for $\t_0=\a\neq\b=\t_1$, an asymmetric positive solution of \eqref{1.1} exists. In such a case, by the symmetry of the problem \eqref{1.1}, a second asymmetric positive solution arises just by taking $\t_0=\b$ and $\t_1=\a$. Thus, the asymmetric positive solutions of \eqref{1.1} appear always by pairs and these pairs of asymmetric solutions are, by construction, symmetric with respect to $x=0.5$. Figure~\ref{fig6} shows the three arcs of trajectory of a positive solution of \eqref{1.1} under the assumption \eqref{5.2}.

\begin{figure}[htbp]
	\centering
	\includegraphics[scale=0.16]{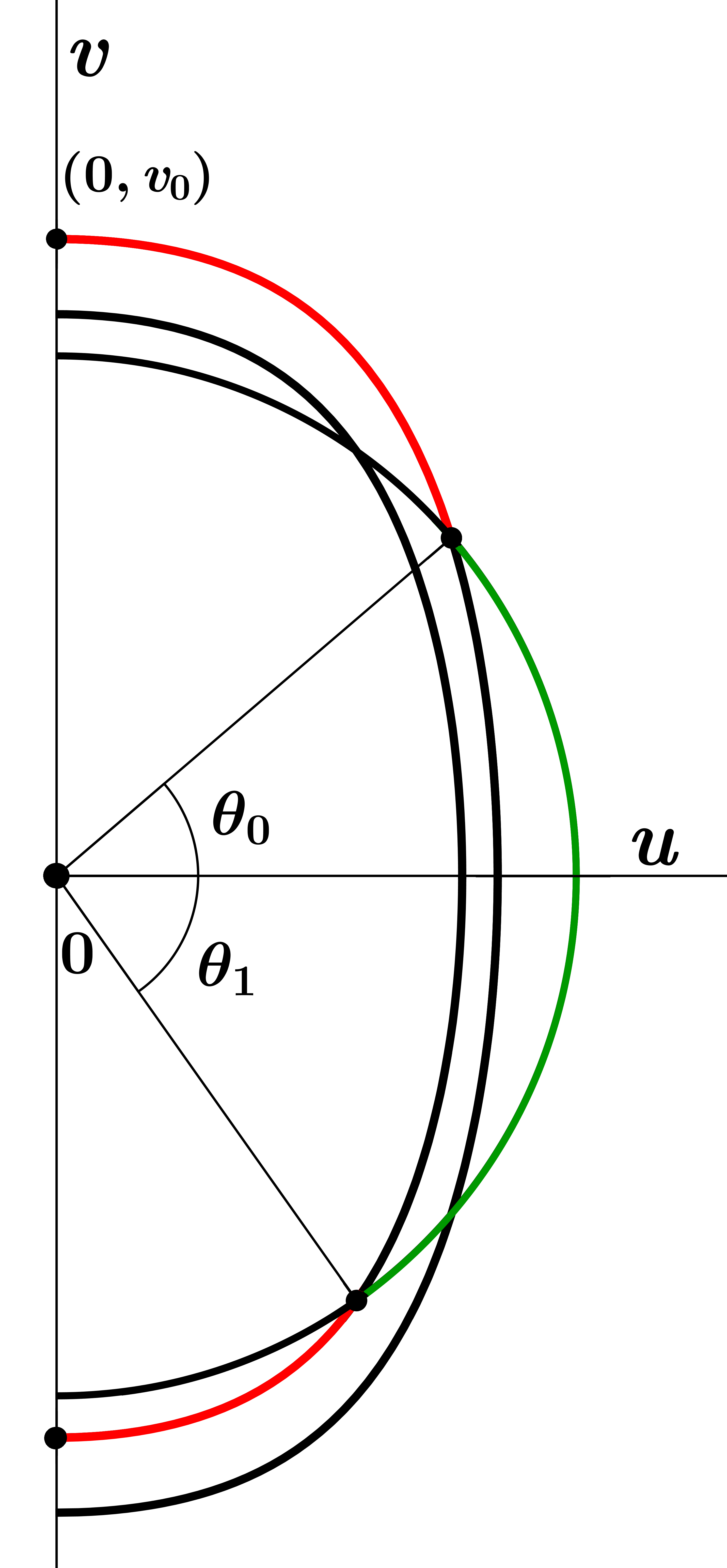}
	\caption{An asymmetric positive solution of \eqref{1.1} ($\t_0\neq\t_1$) satisfying \eqref{5.2}.}
	\label{fig6}
\end{figure}

In order to determine the asymmetric solutions, we need to study the function $S_{\mathscr{N}}(\l,\|u\|_{\infty},\t)$ for a fixed $\l>0$ and  $\theta\in (0,\pi/2)$. To simplify the notations, for any positive solution $u$ of \eqref{1.1}, we denote
$$
   R:=\|u\|_{\infty}>0,
$$
and consider the two-variables re-scaled function
\begin{equation}
\label{5.9}
\begin{split}
\phi(R,\t)&:=2\sqrt{\l}\,S_{\mathscr{N}}(\l,R,\t)\\&=2\int_{0}^{1}\frac{ds}{\sqrt{\frac{1}{\cos^2\t} -s^2 +
\frac{2R^{p-1} \cos^{p-1}\t}{\lambda(p+1)}\left(1-s^{p+1}\right)}}\\&=2R^{-\frac{p-1}{2}}\int_{0}^{1}\frac{ds}{\sqrt{\frac{1-s^2\cos^2\t}{R^{p-1}\cos^2\t}+
\frac{2 \cos^{p-1}\t}{\lambda(p+1)}\left(1-s^{p+1}\right)}}.
\end{split}
\end{equation}

The next lemma collects some basic properties of $\phi(R,\t)$.

\begin{lemma}
\label{le5.1}
The continuous function
$\phi:\,(0,+\infty)\times[0,\pi/2)\rightarrow(0,+\infty)$ defined by \eqref{5.9} is of class $\mc{C}^1$ in $(0,+\infty)\times(0,\pi/2)$ and it satisfies:
\begin{itemize}
\item[\rm (i)] it is decreasing with respect to $R>0$,
\begin{equation*}
\lim_{R\uparrow+\infty}\phi(R,\theta) =0, \;\;\hbox{and}\quad \phi(R,\theta) <
\lim_{R\downarrow 0}\phi(R,\theta) = {\pi} - 2\theta,
\end{equation*}
\item[\rm (ii)] it can be extended by continuity to $\pi/2$ by taking $\phi(\cdot,\pi/2)=0$,
\item[\rm (iii)] it has a continuous derivative with respect to $\t\in(0,\pi/2]$, and
\begin{equation*}
\frac{\p\phi}{\p\theta}(\cdot,\pi/2)=-2.
\end{equation*}
\end{itemize}
\end{lemma}

\begin{proof}
First observe that $\phi$ is also defined for $\t=0$ and $R>0$, with
$$
\phi(R,0)=2
\int_0^1\frac{ds}{\sqrt{1 -s^2 +\frac{2 R^{p-1}}{\lambda(p+1)}\left(1-s^{p+1}\right)}}<2\arcsin 1=\pi.
$$
Moreover, as the integrand is continuous in its domain of definition, $\phi$ is of class  $\mc{C}^1$. The fact that $\phi(R,\t)$ is positive is obvious by definition. Moreover, since $p>1$, it follows from \eqref{5.9} that, for every $\t\in[0,\pi/2)$, the mapping $R\mapsto\phi(R,\t)$ is decreasing and satisfies
$$
  \lim_{R\uparrow+\infty}\phi(R,\theta) =0.
$$
Also, for every $\t \in [0,\frac{\pi}{2})$,
$$
\lim_{R\downarrow 0}\phi(R,\t)
=2\int_0^1\frac{ds}{\sqrt{\frac{1}{\cos^2\t}-s^2}}=2\int_0^{\cos\t}\frac{d\xi}{\sqrt{1-\xi^2}}=\pi-2\t.
$$
Actually, since $\phi(R,\t)$ is uniformly bounded above by $\pi-2\t$, its asymptotic profile as $R\uparrow+\infty$ is given by
\begin{equation}
\label{5.10}
\phi(R,\t) \!\approx\! R^{-\frac{p-1}{2}}\!\!\left(\tfrac{2\l(p+1)}{\cos^{p-1}\t}\right)^{\frac{1}{2}}\!\!\!\int_{0}^{1} \!\!\! \frac{ds}{\sqrt{1-s^{p+1}}}\!=\!
R^{-\frac{p-1}{2}}\!\!\left(\tfrac{2\l(p+1)}{\cos^{p-1}\t}\right)^{\frac{1}{2}}\!\! B(\tfrac{1}{p+1},\tfrac{1}{2}),
\end{equation}
where $B$ is the Euler beta-function (see \cite[Section 7.1]{LGMHZ-2023}). This ends the proof of Part (i). Part (ii) holds because
$$
\lim_{\t\uparrow \tfrac{\pi}{2}}\phi(\cdot,\t)=0.
$$
Finally, Part (iii) follows from the fact that, for every $\t\in(0,\pi/2)$,
\begin{equation}
\label{5.11}
\frac{\p\phi}{\p\t}(R,\t)\!=\! 2R^{-\frac{p-1}{2}}\sin\t\!\!\bigintsss_{0}^{1}\frac{\frac{2(p-1)\cos^{p+1}\t}{\l(p+1)}(1-s^{p+1})-\frac{1}{R^{p-1}}}{\left(\frac{1-s^2\cos^2\t}{R^{p-1}}+\frac{2 \cos^{p+1}\t}{\lambda(p+1)}\left(1-s^{p+1}\right)\right)^{3/2}}\,ds.
\end{equation}
Indeed, for every $R>0$, $\frac{\p\phi}{\p\theta}(R,\pi/2)=-2$. This ends the proof.
\end{proof}

Figure~\ref{Fig7} left shows the graphs of the map $\t\mapsto \phi(R,\t)$ for a series of increasing values of $R>0$, $R_1<R_2<R_3$. For sufficiently small $R>0$  the map is decreasing, while, for larger values of $R$,  there is a point of maximum in the interval $(0,\pi/2)$. To generate the figure, we have chosen $p=3,$ $\l=6$ and $R_1=4,$ $R_2=20,$ $R_3=80$, respectively. Note that all the graphs have a common tangent
at $\pi/2$ as proved in Lemma \ref{le5.1} (iii). The right plot of Figure~\ref{Fig7} shows the
graph of the function $\phi(\cdot,0)$ for $p=3$ and $\lambda=6$.

\begin{figure}[htbp]
	\centering
	\includegraphics[scale=0.28]{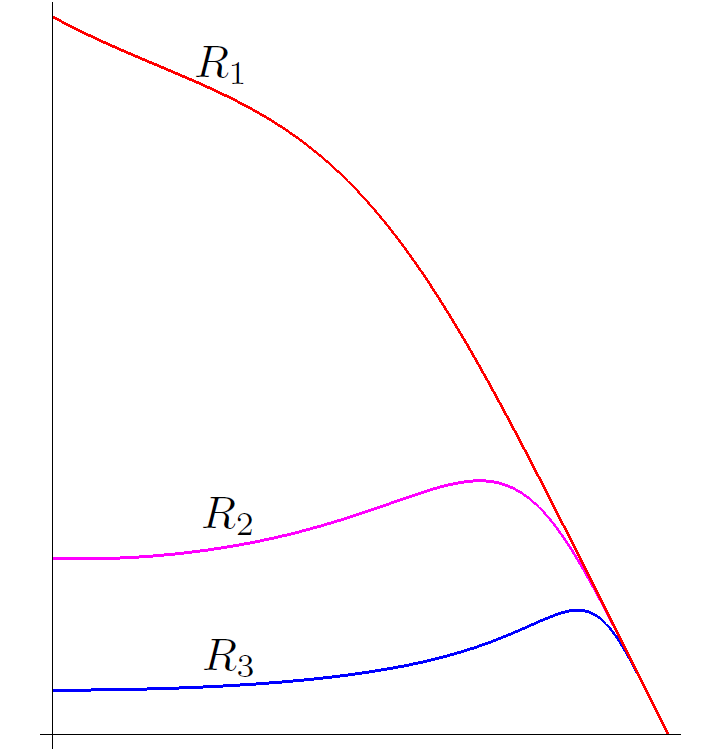}\quad
	\includegraphics[scale=0.2]{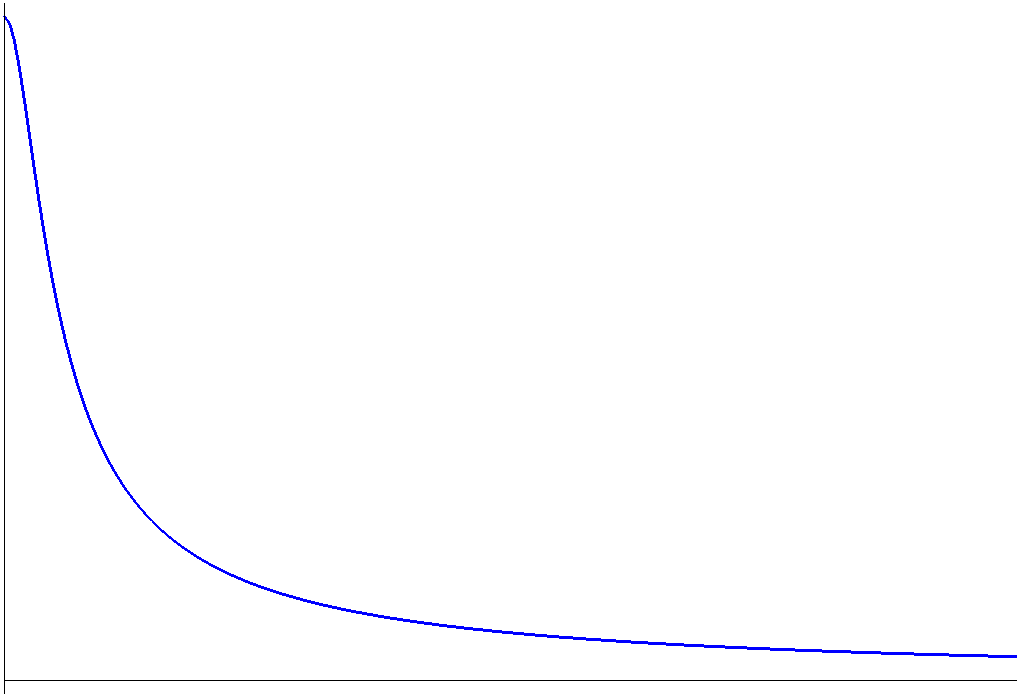}
	\caption{Graphs of the maps $\theta\mapsto \phi(R_i,\theta)$ for $i\in\{1,2,3\}$ (left), and graph of the map $R\mapsto\phi(R,0)$ on the interval $[0,80]$ (right).}
	\label{Fig7}
\end{figure}

On the other hand, when trying to determine the limit of the derivative as $\t\downarrow 0$ one should be aware that the integral
\[
\bigintsss_{0}^{1}\frac{\frac{2(p-1)}{\l(p+1)}(1-s^{p+1})-\frac{1}{R^{p-1}}}{\left(\frac{1-s^2}{R^{p-1}}+\frac{2}{\lambda(p+1)}\left(1-s^{p+1}\right)\right)^{3/2}}\,ds
\]
is divergent at $s=1$ and, hence, \eqref{5.11} takes an indeterminate form. So, we always assume $\theta>0$ when analyzing \eqref{5.11}. For every $R>0$ and $\t\in(0,\pi/2]$, it is clear that the sign of \eqref{5.11} is the sign of the integrand's numerator.
\par
According to Theorem \ref{th3.1}, we already know that
$$
   \lim_{h\uparrow 1}R=\lim_{h\uparrow 1}\|u\|_\infty=\infty.
$$
Thus, we need a description of the function $\phi(R,\cdot)$ for large $R$, which is given by
the next result.

\begin{lemma}
\label{le5.2}
Given $a,b$ with $0<a<b<\pi/2$, there exists $R^*=R^*(a,b)$ such that, for every $R>R^*$, the following holds:
\begin{enumerate}
\item[\rm(a)] The function $\phi(R,\cdot)$ is increasing in the interval $[a,b]$.
\item[\rm(b)] The function $\phi(R,\cdot)$ has a point of maximum at $\t^*\in(b,\pi/2)$. Moreover, it is the absolute maximum point.
\item[\rm(c)] There exists $\t_*\in[0,a]$ such that
\[
\phi(R,\t_*)=\min_{\t\in[0,a]}\phi(R,\t)>0=\phi(R,\pi/2).
\]
\end{enumerate}
\end{lemma}

\begin{proof}
First, note that, fixing  $\t\in[a,b]$, the map $R\mapsto\frac{\p\phi}{\p\t}(R,\t)$ is decreasing and converges to zero as $R\uparrow +\infty$. Indeed, as  $R\uparrow+\infty$, we have that
\begin{equation}
\label{5.12}
\begin{split}
\frac{\p\phi}{\p\t}(R,\t)&\approx R^{-\frac{p-1}{2}}(p-1)\sin\t\left(\tfrac{2\l(p+1)}{\cos^{p+1}\t}\right)^{\frac{1}{2}}\int_{0}^{1} \frac{ds}{(1-s^{p+1})^{\frac{1}{2}}}\\&=R^{-\frac{p-1}{2}}(p-1)\sin\t\left(\tfrac{2\l(p+1)}{\cos^{p+1}\t}\right)^{\frac{1}{2}}B(\tfrac{1}{p+1},\tfrac{1}{2}),
\end{split}
\end{equation}
which provides us with the derivative of \eqref{5.10}. By \eqref{5.12}, there exists $R^*$ such that $\frac{\p}{\p\t}\phi(R,\t)>0$ for every $R>R^*$ and $\t\in[a,b]$, proving Part (a).
\par
Now, by Part (a) and recalling that $\phi(R,\pi/2)=0$, the existence of a point of maximum at $(b,\pi/2)$ holds, which proves Part (b). Moreover, by \eqref{5.10} and \eqref{5.12}, it becomes apparent that, for sufficiently large $R$, $\t^*$ is a point of absolute maximum. Finally, Part (c) follows by the continuity of $\phi(R,\t)$. This ends the proof.
\end{proof}

Thanks to  Lemma \ref{le5.2}, for every $\lambda >0$ and $\varepsilon>0$, there exists $R^*>0$
such that, for every $R>R^*$, $\phi(R,\cdot)$ is increasing in $[\varepsilon,\frac{\pi}{2}-\varepsilon]$ and all the critical points of $\phi(R,\cdot)$
lie on $[0,\varepsilon)\cup(\frac{\pi}{2}-\varepsilon,\frac{\pi}{2})$. For such $R$'s,  we set
$$
    \phi_M=\phi_M(R):= \max_{\theta\in [0,\tfrac{\pi}{2}]} \phi(R,\theta).
$$
By Lemma \ref{le5.2} (b), such a value is achieved at some $\t\in [\pi/2-\varepsilon,\pi/2]$. Accordingly, we set
$$
\theta_M=\theta_M(R,\varepsilon):=\min\{\theta\in [\pi/2-\varepsilon,\pi/2):\;\phi(R,\theta)=\phi_M\}.
$$
We also denote by
$\t_m= \t_m(R,\varepsilon)\in [0,\varepsilon]$ the maximum value of $\t$ for which $\phi(R,\t)$ has a minimum in $[0,\varepsilon]$, taking $\t_m=0$ if there are no critical points
in $[0,\varepsilon]$. Moreover, we set
$$
\phi_m  =\phi_m(R):=\phi(R,\theta_m)>0, \;\; \bar{\theta}  :=
\min\{\theta\in [\pi/2-\varepsilon,\pi/2):\;
\phi(R,\theta)=\phi_m\}.
$$
With these notations, by Lemma \ref{le5.2}, the next result holds.

\begin{proposition}
\label{prop5.1}
For every  $R>R^*$, the restriction of the function $\phi(R,\cdot)$ to the interval $[\theta_m,\bar{\theta}]$ satisfies the following conditions:
\begin{itemize}
\item[\rm (a)] $\phi(R,\theta_m) = \phi(R,\bar{\theta}) = \phi_m >0$, and $\phi(R,\theta) > \phi_m$ for all $\theta\in (\theta_m,\bar{\theta}).$
\item[\rm (b)] $\phi(R,\cdot)$ is increasing in $[\theta_m,\frac{\pi}{2}-\varepsilon]$.
\item[\rm (c)] $\theta_M\in [\frac{\pi}{2}-\varepsilon,\frac{\pi}{2})$ with
$$
   \phi(R,\theta_M)=\max_{\t\in[\t_m,\t_M]}\phi(R,\t) \geq\phi(R,\tfrac{\pi}{2}-\varepsilon)> \phi_m.
$$
\end{itemize}
\end{proposition}

There is numerical evidence that, for large $R>0$, there is a unique point of maximum
for $\phi(R,\cdot)$, $\t_M$,  and that $\phi(R,\cdot)$ is decreasing in
$[\theta_M,\pi/2]$, i.e., $\t_M$ is the unique $\theta\in [\pi/2-\varepsilon,\pi/2)$ such that
$\phi(R,\theta)=\phi_m$. However, since this assertion has not been proved in the analysis of the function $\phi(R,\cdot)$, we are forced to make the definitions of
$\theta_M$ and $\bar{\theta}$ a bit more intricate.
\par
Figure~\ref{Fig8} shows the main features of Lemma \ref{le5.2} and Proposition \ref{prop5.1}.

\begin{figure}[htbp]
\centering
\includegraphics[scale=0.23]{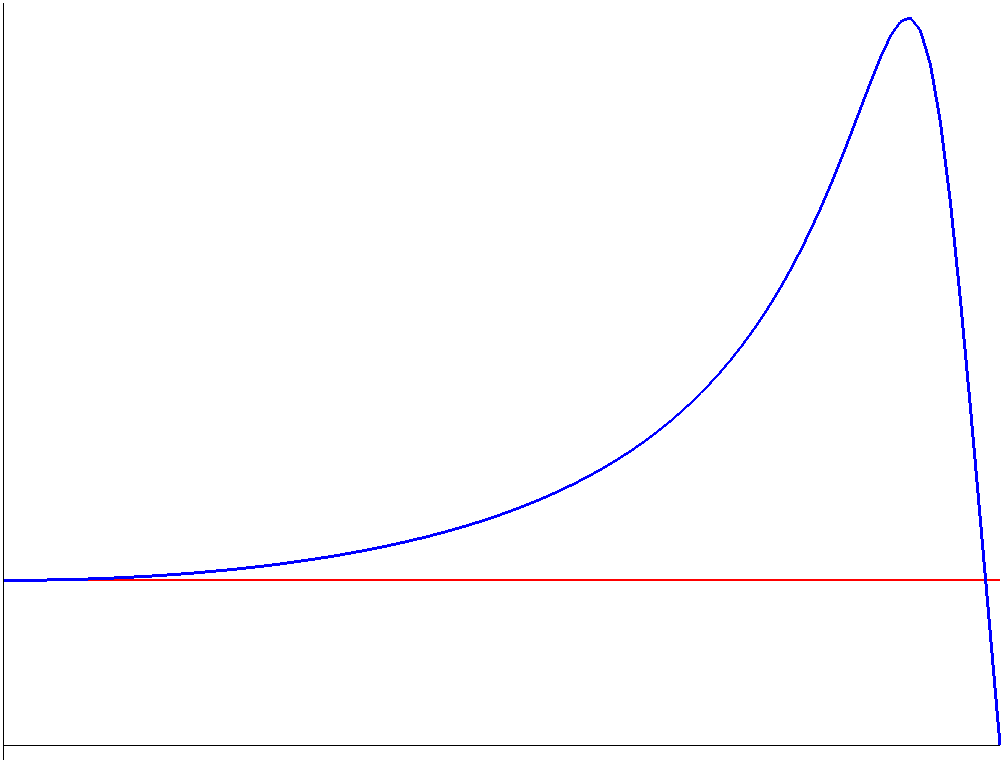}\qquad
\includegraphics[scale=0.2]{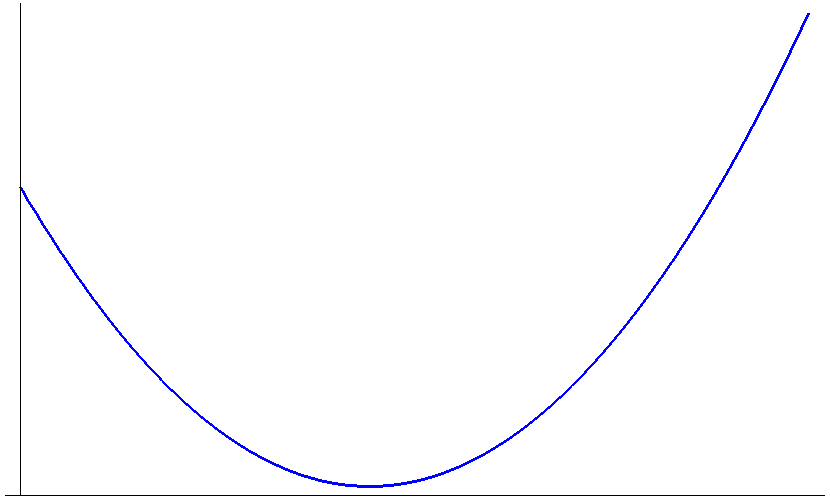}
\caption{Graph of $\theta\mapsto \phi(R,\cdot)$ for large $R$ (left), and graph of $\theta\mapsto \phi(R,\theta)$	for large $R$ in a right neighborhood of $\theta=0$ (right).}
\label{Fig8}
\end{figure}

\noindent

The left plot of Figure~\ref{Fig8} provides a genuine graph of $\phi(R,\cdot)$ for large $R$. Apparently, the  figure shows the presence of a unique (maximum) critical point in the interval $(0,\pi/2)$; the red horizontal line crossing the graph at the points $(\theta_m,\phi_m)$ and $(\bar{\theta},\phi_m)$
at the level $y=\phi_m$. As this figure might suggest that $\phi(R,\cdot)$
is increasing in the whole interval $[0,\theta_M]$, the plot on the right shows a zoom of the graph
of $\phi(R,\cdot)$ in a neighborhood of $\t=0$, with a minimum, $\theta_m$,  close to $\t=0$. These graphs have been computed for $p=3$ and $\l=6$ with $R=300$ and $\t\in[0,0.05]$. In such case, $\phi(300,\cdot)\in[0.15116,0.15122]$.
\par
As, for every $h \in (1 - \tfrac{\phi_M}{\sqrt{\lambda}},
1 - \tfrac{\phi_m}{\sqrt{\lambda}})$,  the equation
$$
    \phi(R,\theta) = (1-h)\sqrt{\lambda}, \qquad \theta\in (0,\pi/2),
$$
has, at least, two solutions:
one with $\t\in(\t_m,\t_M)$ and a second one with $\t\in(\t_M,\bar{\theta})$, we can continuously parameterize these solutions as follows.

\begin{lemma}
\label{le5.3}
There exists a continuous map $\gamma=(\gamma_1,\gamma_2): [0,1]\to {\mathbb R}^2$ such that
\begin{enumerate}
\item[{\rm (a)}] $\gamma_1(s)\in (\theta_m,\theta_M)$ and $\gamma_2(s)\in (\theta_M,\bar{\theta})$ for all $s\in (0,1)$,
\item[{\rm (b)}] $\gamma_{1}(0)=\theta_m$, $\gamma_{2}(0)=\bar{\theta}$,  and $\gamma_{1}(1)= \theta_M=\gamma_{2}(1)$,
\item[{\rm (c)}] $\phi(R,\gamma_1(s))=\phi(R,\gamma_2(s))$ for all $s\in [0,1]$.
\end{enumerate}
\end{lemma}
\noindent

Lemma \ref{le5.3} is just a special version of the Mountain Climbing Theorem, applied to the function $\phi(R,\cdot)$. There are different versions of this theorem which either
require that the function under consideration is piecewise monotone, as in
\cite{Wh-1966}, or that it has no intervals where it is constant, as in \cite{Ho-1952,Ke-1993}. The function $\phi(R,\cdot)$ is real analytic in a closed sub-interval $J$ of $(0,\pi/2)$
with $[\theta_m,\bar{\theta}]\subset {\rm int}J$. Therefore, $\phi(R,\cdot)$ must have finitely many
critical points in $[\theta_m,\bar{\theta}]$ and the Mountain Climbing Theorem can be applied.
\par
The geometric interpretation of Lemma \ref{le5.3} goes as follows. There are two climbers,
$(\gamma_1(s),\phi(R,\gamma_1(s)))$ and $(\gamma_2(s),\phi(R,\gamma_2(s)))$,
aiming to attain the highest  pick of a  mountain at $\theta=\theta_M$,
climbing continuously on the two sides
$$
   \{(\theta,\phi(R,\theta)): \theta\in [\theta_m,\theta_M)\} \quad\text{and }\quad
   \{(\theta,\phi(R,\theta)): \theta\in (\theta_M,\bar{\theta}]\}
$$
of the graph of $\phi(R,\cdot)$. The two climbers
start at the time $s=0$ at the same level, $\phi_m$, though at opposite sides
of the peak, and both meet at the summit, $(\theta_M,\phi_M)$, at the time $s=1$. While climbing, they maintain the same altitude at each time. It is apparent that if $\theta_M\in (\theta_m,\pi/2)$ is the unique point of maximum of $\phi(R,\cdot)$ for $R$ large (see  Figure \ref{Fig8} and the discussion after Proposition \ref{prop5.1}), then, one can take the functions
$$
    \eta_1:=\bigl(\phi(R,\cdot)\vert_{[\theta_m,\theta_M]}\bigr)^{-1}, \quad
   \eta_2:=\bigl(\phi(R,\cdot)\vert_{[\theta_M,\bar{\theta}\,]}\bigr)^{-1},
$$
defined on $[\phi(R,0),\phi_M]$, and re-parameterize them to $[0,1]$,
by setting
$$
  \gamma_1(s):= \eta_i\bigl(\phi_m + s(\phi_M-\phi_m)\bigr), \quad s\in [0,1], \quad (i=1,2),
$$
without invoking the Mountain Climbing Theorem. Using this theorem allows us to avoid a further analysis of $\frac{\p\phi}{\p\t}(R,\t)$ in $[\pi-\varepsilon,\frac{\pi}{2})$ in order to prove rigorously that $\theta_M$ is the unique local maximum of  $\phi(R,\cdot)$ for sufficiently large $R>0$. As it can be proved that, for some $\t(R)\in(\t_M,\pi/2)$,
$$
\lim_{R\uparrow +\infty}\frac{\p\phi}{\p\t}(R,\t(R))=+\infty
=\lim_{R\uparrow +\infty}\frac{\p^2\phi}{\p\t^2}(R,\t(R))
$$
the underlying analysis can become quite hard due to the lack of information on the sign of the derivatives near $\pi/2$.
\par
Now, we are ready to prove the existence of asymmetric positive solutions for \eqref{1.1}.

\begin{theorem}
\label{th5.1}
For any given $\lambda\in \left(\pi^2/4,\pi^2\right)$, there exists $R_{\lambda}>0$ such that, for every $R>R_{\lambda}$, the problem \eqref{1.1} has, at least, two asymmetric positive solutions for a suitable $h\in (0,1)$. Moreover, $\lim_{R\uparrow +\infty}h=1$.
\end{theorem}
\begin{proof}
As we have previously remarked, the asymmetric positive solutions come in pairs. Thus, it suffices to show the existence of, at least, one asymmetric solution. In terms of
$$
   \phi(R,\t)\equiv 2\sqrt{\l}\,S_{\mathscr{N}}(\l,R,\t),
$$
\eqref{5.8} can be, equivalently expressed, as
\begin{equation}
\label{v.13}
\begin{cases}
\theta_0 + \theta_1 = h\sqrt{\l},\\[1pt]
\phi(R,\theta_0) = (1-h)\sqrt{\l}=\phi(R,\theta_1).
\end{cases}
\end{equation}
By Lemma \ref{le5.3} with  $\theta_0=\gamma_1(s)$ and $\theta_1=\gamma_2(s)$,
we have that, for every $s\in(0,1)$,
\begin{equation}
\label{5.14}
\phi(R,\gamma_1(s)) =
\phi(R,\gamma_2(s)) =(1-h(s))\sqrt{\l}\in (\phi_m,\phi_M).
\end{equation}
Thus, we need only to solve the first equation of \eqref{v.13}, which reads as
\begin{equation}
\label{5.15}
\gamma_1(s) + \gamma_2(s) + \phi(R,\gamma_1(s)) = \sqrt{\lambda}
\end{equation}
by \eqref{5.14}. Equivalently, $g(s)=\sqrt{\l}$, where we are setting
$$
   g(s):= \gamma_1(s) + \gamma_2(s) + \phi(R,\gamma_1(s)),\qquad s\in [0,1].
$$
Since $\phi(R,\t)<\pi-2\t$ for every $R>0$, we find from Lemma \ref{le5.3} that
\begin{align*}
		\lim_{s\downarrow 0} g(s) =
		\theta_m + \bar{\theta} + \phi(R,\theta_m) & = \theta_m +\bar{\theta} + \phi(R,\bar{\theta})\\
		&<\theta_m + \bar{\theta} + {\pi} - 2\bar{\theta}=		\theta_m+ {\pi}-\bar{\theta}.
\end{align*}
Thus, recalling that $\theta_m\leq \varepsilon$ and $\bar{\theta} \geq \pi/2-\varepsilon$, it becomes apparent that
$$
\lim_{s\downarrow 0} g(s)  <\pi/2 + 2\varepsilon,
$$
provided $R>R^*(\e)\equiv R^*(\e,\pi/2-\e)$, where $R^*(\e,\pi/2-\e)$ is the value of $R$ given by Lemma \ref{le5.2}. As we are assuming that $\sqrt{\lambda} > \pi/2$, we can  make the choice
$$
    2\varepsilon \in (0,\sqrt{\l}-\pi/2).
$$
This ensures that, for every $R>R^*(\e)$,
\begin{equation}
\label{5.16}
		\lim_{s\downarrow 0} g(s) < \sqrt{\lambda}.
\end{equation}
Thanks again to Lemma \ref{le5.3},
$$
	\lim_{s\uparrow 1} g(s) =2\theta_M + \phi(R,\theta_M) = 2\theta_M+\phi_M.
$$
Since $\sqrt{\l}<\pi$ and $2\theta_M\geq \pi -2\varepsilon$, if  $\varepsilon>0$ is chosen so that
$$
\sqrt{\lambda} < \pi -2\e,
$$
then, for every $R>R^{*}(\varepsilon)$, we find that
$$
    \lim_{s\uparrow 1} g(s) \geq \pi -2\varepsilon > \sqrt{\lambda},
$$
besides \eqref{5.16}. Note that $\e$ has been chosen so that
$$
   \tfrac{\pi}{2} + 2\e <\sqrt{\l}<\pi-2\e.
$$
Therefore, by the continuity of $g(s)$,  there is $\hat{s}\in (0,1)$ such that
$g(\hat s)= \sqrt{\lambda}$. By construction, $\gamma_1(\hat{s}) < \theta_M < \gamma_1(\hat{s})$, so that we can take $\theta_0=\gamma_1(\hat{s})$ and $\theta_1=\gamma_2(\hat{s})$, or $\theta_0=\gamma_2(\hat{s})$ and
$\theta_1=\gamma_1(\hat{s})$, to get the desired asymmetric solutions.
Finally, from \eqref{5.14}, we find that
\begin{equation}
\label{5.17}
    \phi(R,\gamma_1(\hat{s})) =
	\phi(R,\gamma_2(\hat{s}))=(1-h(\hat{s}))\sqrt{\l}\in(\phi_m,\phi_M),
\end{equation}
providing us with the precise value of $h$ for which we can guarantee the existence of two asymmetric solutions,
\begin{equation}
\label{5.18}
    h(\hat{s})= 1 -\frac{\phi(R,\g_1(\hat{s}))}{\sqrt{\lambda}}.
\end{equation}
Observe that, thanks to \eqref{5.17},
$$
\lim_{R\uparrow +\infty}(1-h(\hat s))\sqrt{\l} \leq \lim_{R\uparrow +\infty}\phi_M:=\lim_{R\uparrow +\infty}\phi(R,\t_M)=0.
$$
Consequently,
\begin{equation}
\label{5.19}
   \lim_{R\uparrow +\infty}h(\hat s)=1.
\end{equation}
This ends the proof.
\end{proof}

Note that, recalling that $R=\|u\|_{\infty}$,  \eqref{5.19} is a counterpart of \eqref{3.1}.


\begin{thebibliography}{xx}
	
\bibitem{ALG} H. Amann and J. L\'opez-G\'omez, A priori bounds and multiple solutions for
	superlinear indefinite elliptic problems, \emph{J. Differential Equations}
	\textbf{146} (1998), 336--374.
	
	
\bibitem{CLGT-2024} P. Cubillos, J. L\'{o}pez-G\'{o}mez and A. Tellini, High multiplicity of positive solutions in a superlinear problem of Moore-Nehari type, \emph{Commun. Nonlinear Sci. Numer. Simul.} \textbf{136} (2024),  27p.
	
\bibitem{FLG-2022} M. Fencl and  J. L\'{o}pez-G\'{o}mez, Global bifurcation diagrams of positive solutions for a class of 1D superlinear indefinite problems, \emph{Nonlinearity} \textbf{35} (2022), 1213--1248.
	
\bibitem{Ho-1952} J. T. Homma, A theorem on continuous functions, \emph{Kodai Math. Semin. Rep.} \textbf{1} (1952), 13--16.
	
\bibitem{Ka-2012} R. Kajikiya, Non-even least energy solutions of the {E}mden-{F}owler
	equation, \emph{Proc. Amer. Math. Soc.} \textbf{140} (2012), 1353--1362.
	
\bibitem{Ka-2014} R. Kajikiya,  Three positive solutions of the one-dimensional generalized
	{H}\'{e}non equation, \emph{Results Math.} \textbf{66} (2014), 427--459.
	
\bibitem{Ka-2022} R. Kajikiya, Bifurcation of symmetric solutions for the sublinear Moore-Nehari differential equation, \emph{Journal of Math. Anal. and Appl.} \textbf{512} (2022), 35p.

\bibitem{Ka-2023} R. Kajikiya, Bifurcation of nodal solutions for the {M}oore-{N}ehari differential equation, \emph{NoDEA Nonlinear Differential Equations Appl.} \textbf{30}(8) (2023), 29p.
	
\bibitem{Ka-2018} R. Kajikiya, I. Sim and S. Tanaka, Symmetry-breaking bifurcation for the {M}oore-{N}ehari differential equation, \emph{NoDEA Nonlinear Differential Equations Appl.} \textbf{25} (2018), 22p.

	
\bibitem{Ke-1993} T. Keleti,  The mountain climbers' problem, \emph{Proc. Amer. Math. Soc.} \textbf{117} (1993), 89--97.
	
\bibitem{LG-2013} J. L\'{o}pez-G\'{o}mez,  \emph{Linear Second Order Elliptic Operators},
	World Scientific, Singapore, 2013.

\bibitem{LG-2015} J. L\'{o}pez-G\'{o}mez, \emph{Metasolutions of Parabolic Equations in
Population Dynamics}, CRC Press, Boca Raton, 2015.
	

	
\bibitem{LGMHZ-2023} J. L\'{o}pez-G\'{o}mez, E. Mu\~{n}oz-Hern\'{a}ndez and F. Zanolin, Rich dynamics in planar systems with heterogeneous nonnegative 	weights, \emph{Commun. Pure Appl. Anal.} \textbf{22} (2023), 1043--1098.

\bibitem{LGTZ-2014} J. L\'{o}pez-G\'{o}mez, A. Tellini and F. Zanolin, High multiplicity and complexity of
the bifurcation diagrams of large solutions for a class of superlinear indefinite problems, \emph{Comm. Pure and Appl. Analysis} \textbf{13} (2014), 1--73.

\bibitem{LGRab-2016} J. L\'{o}pez-G\'{o}mez and P. H. Rabinowitz, The effects of spatial heterogeneities on
some multiplicity results, \emph{Disc. Cont. Dyn. Systems} \textbf{36} (2016), 941--952.
	
\bibitem{MoNe-1959} R. A. Moore and Z. Nehari, Nonoscillation theorems for a class of nonlinear differential
	equations, \emph{Trans. Amer. Math. Soc.} \textbf{93} (1959), 30--52.
	
\bibitem{Ta-2009} S. Tanaka, An identity for a quasilinear {ODE} and its applications to
	the uniqueness of solutions of {BVP}s, \emph{J. Math. Anal. Appl.} \textbf{351} (2009), 206--217.
	
\bibitem{Ta-2013} S. Tanaka,  Morse index and symmetry-breaking for positive solutions of
	one-dimensional {H}\'{e}non type equations, \emph{J. Differential Equations} \textbf{255} (2013), 1709--1733.
	
\bibitem{Wh-1966} J. V. Whittaker, A mountain-climbing problem, \emph{Canadian J. Math.} \textbf{18} (1966), 873--882.	
\end{thebibliography}
\end{document}